\documentclass[12pt]{article}
\usepackage{amsfonts}
\usepackage{amssymb}
\usepackage{mathrsfs}
\usepackage{srcltx}
\usepackage{amsmath}
\usepackage{amsthm}
\usepackage{amstext}
\usepackage{amsopn}
\usepackage{graphicx}

\newtheorem{theorem}{Theorem}[section]
\newtheorem{lemma}[theorem]{Lemma}

\theoremstyle{definition}

\newtheorem{example}[theorem]{Example}

\theoremstyle{remark}
\newtheorem{remark}[theorem]{Remark}

\numberwithin{equation}{section}

\newcommand{\ba}{\begin{array}}
\newcommand{\ea}{\end{array}}
\newcommand{\f}{\frac}

\newcommand{\la}{\lambda}

\newcommand{\N}{{\mathbf N}}
\newcommand{\ds}{\displaystyle}

\begin{document}
\date{}
\title{ \bf\large{Spatial nonhomogeneous periodic solutions induced by nonlocal prey competition in a diffusive predator-prey model}\footnote{This research is supported by the National Natural Science Foundation of China (Nos. 11771109 and 11371111)}}
 \author{Shanshan Chen\footnote{Corresponding Author, Email: chenss@hit.edu.cn},\ \  Junjie Wei\footnote{Email: weijj@hit.edu.cn},\ \ Kaiqi Yang
 \\
{\small  Department of Mathematics,  Harbin Institute of Technology,\hfill{\ }}\\
 {\small Weihai, Shandong, 264209, P.R.China.\hfill{\ }}}
\maketitle

\begin{abstract}
{The diffusive Holling-Tanner predator-prey model with no-flux boundary conditions and nonlocal prey competition is considered in this paper. We show the existence of spatial nonhomogeneous periodic solutions, which is induced by nonlocal prey competition. In particular,
the constant positive steady state can lose the stability through Hopf bifurcation when the given parameter passes through some critical values, and the bifurcating periodic solutions near such values can be spatially nonhomogeneous and orbitally asymptotically stable. This phenomenon is different from that in models without nonlocal effect. }

\noindent {\bf{Keywords}}: Predator-prey model; Nonlocal competition; Hopf bifurcation; Spatial nonhomogeneous periodic solutions.
\end{abstract}

\section {Introduction}
During the past twenty years, bifurcations and spatiotemporal patterns for homogeneous reaction-diffusion equations have been studied extensively, see
\cite{Dongyy2017,Guo2017,LiWang,LiaoWang,PengYi,ShiR,SongZou,WangJF,WangSW,Yi,ZhouShi} and references therein.
In particular, spatially homogeneous and nonhomogeneous periodic orbits can occur through Hopf bifurcation. To our knowledge, for homogeneous reaction-diffusion equations,
the constant positive steady state can lose the stability when the given parameter passes through some Hopf bifurcation values, and the bifurcating periodic solutions near such values are always spatially homogeneous. The spatially nonhomogeneous periodic orbits can also occur through Hopf bifurcation, but they are always unstable. This phenomenon was firstly obtained by Yi et al. \cite{Yi} for
the following diffusive predator-prey model with Holling type-II predator functional response,
\begin{equation}\label{Yi}
\begin{cases}
  \ds\frac{\partial u}{\partial t}-d_1\Delta u=u\left(1-\ds\f{u}{k}\right)-\ds\frac{muv}{1+u}, & x\in \Omega,\; t>0,\\
 \ds\frac{\partial v}{\partial t}-d_2\Delta v=-\theta v+\ds\frac{muv}{1+u}, & x\in\Omega,\; t>0,\\
 \partial_\nu  u=\partial_\nu
  v=0,& x\in \partial \Omega,\;
 t>0.\\
\end{cases}
\end{equation}
Due to the instability, it is hard to obtain spatially nonhomogeneous periodic orbits numerically for homogeneous reaction-diffusion equations.

It has been pointed out that there is no real justification for assuming that
the interaction between individuals of a species is local, and the individuals at different locations may compete for common resource or communicate either visually or by chemical means \cite{DuHsu2010,Furter}. Models with nonlocal competition effect have been studied extensively, see \cite{Berestycki,Billingham,fang2011monotone,Gregory,gourley2000travelling,Hamel} for results on traveling wave solutions and
\cite{Alves,Corr,Furter,Sun,Yamada} for existence and bifurcations of steady states.
Recently, considering nonlocal competition of prey, Merchant and Nagata \cite{Merchant} proposed the following nonlocal Rosenzweig-MacArthur predator-prey model,
\begin{equation}\label{Yi-nonlocal}
\begin{cases}
  \ds\frac{\partial u}{\partial t}=d_1\Delta u+au\left(1-\ds\f{1}{k}\int_{\Omega} K(x,y)u(y,t)dy\right)-\ds\frac{buv}{u+m}, & x\in \Omega,\; t>0,\\
 \ds\frac{\partial v}{\partial t}=d_2\Delta v-dv+\ds\frac{cuv}{u+m}, & x\in\Omega,\; t>0,\\
\end{cases}
\end{equation}
and nonlocal Holling-Tanner predator-prey model,
\begin{equation}\label{Tanner-nonlocal}
\begin{cases}
  \ds\frac{\partial u}{\partial t}=d_1\Delta u+au\left(1-\ds\f{1}{k}\int_{\Omega} K(x,y)u(y,t)dy\right)-\ds\frac{buv}{u+m}, & x\in \Omega,\; t>0,\\
 \ds\frac{\partial v}{\partial t}=d_2\Delta v+cv\left(1-\ds\frac{ev}{u}\right), & x\in\Omega,\; t>0.\\
\end{cases}
\end{equation}
When $\Omega=(-\infty,\infty)$,
they showed that the nonlocal competition can induce complex spatiotemporal patterns.
For one-dimensional bounded domain $(0,\ell\pi)$, Chen and Yu \cite{ChenYuDCDS} chose $K(x,y)=1/\ell\pi$ as in \cite{Furter} and obtained that the constant positive steady state of model \eqref{Yi-nonlocal} can also lose the stability when the given parameter passes through some Hopf bifurcation values, but the bifurcating periodic solutions near such values can be spatially nonhomogeneous. This phenomenon is different from that in model \eqref{Yi} without nonlocal effect. However, the properties of Hopf bifurcation, such as the bifurcation direction and stability of the bifurcating periodic solutions, have not been solved theoretically.

In this paper, we mainly consider model \eqref{Tanner-nonlocal} with nonlocal competition of prey, and show the existence and properties of Hopf bifurcation.
We remark that if $K(x,y)=\delta(x-y)$, then model \eqref{Tanner-nonlocal} is reduced to the classical Holling-Tanner predator-prey model, for which the steady states, Hopf bifurcations and Turing instability were investigated in \cite{LiJiang,MaLi}, and the global stability of the positive constant equilibrium was considered in \cite{chen2012,peng2007,Qi2016}.
As in \cite{ChenYuDCDS,Furter}, we choose $\Omega=(0,\ell\pi)$, $K(x,y)=1/\ell\pi$, and then model \eqref{Tanner-nonlocal} with no-flux boundary conditions and nonnegative initial values takes the following form:
\begin{equation}\label{nonlocal2}
\begin{cases}
  \ds\frac{\partial u}{\partial t}=d_1 u_{xx}+au\left(1-\ds\f{1}{k\ell\pi}\int_{0}^{\ell\pi} u(y,t)dy\right)-\ds\frac{buv}{u+m}, & x\in (0,\ell\pi),\; t>0,\\
 \ds\frac{\partial v}{\partial t}=d_2 v_{xx}+cv\left(1-\ds\frac{ev}{u}\right), & x\in (0,\ell\pi),\; t>0,\\
 u_x(0,t)=
  v_x(0,t)=0,\;u_x(\ell\pi,t)=
  v_x(\ell\pi,t)=0,&
 t>0,\\
 u(x,0)=u_0(x)>0,\;\;v(x,0)=v_0(x)>0,
\end{cases}
\end{equation}
where $u(x,t)$ and $v(x,t)$ stand for the densities
of the prey and predator at time $t$ and location
$x$ respectively, and parameters $a$, $b$, $c$, $e$, $k$, $\ell$, $d_1$ and $d_2$ are all positive constants. Specifically, $\ell$ is the spatial scale; $d_1$ and $d_2$ are the diffusion rates of the prey and predator respectively; $k$ represents the carrying capacity of the prey; $b$ and $e$ measure the
interaction strength between the predator and prey; $a$ and $c$ are the intrinsic growth rates of the prey and predator respectively; and $m$ measures the prey's ability to evade attack, see \cite{tanner1975} for more detailed biological explanation. By using the following rescaling,
$$\tilde t=at,\;\;\tilde u=\ds\f{u}{m},\;\;\tilde v=\ds\f{ev}{m},$$
denoting $$\tilde \beta=\ds\f{ m}{k},\;\;\tilde b=\ds\f{b}{ae},\;\;\tilde c=\ds\f{c}{a},\;\;\tilde d_1=\ds\f{d_1}{a},\;\;\tilde d_2=\ds\f{d_2}{a},$$
and dropping the tilde sign,
system \eqref{nonlocal2} can be simplified as follows:
\begin{equation}\label{nonlocal3}
\begin{cases}
  \ds\frac{\partial u}{\partial t}=d_1 u_{xx}+u\left(1-\ds\f{\beta}{\ell\pi}\int_{0}^{\ell\pi} u(y,t)dy\right)-\ds\frac{buv}{u+1}, & x\in (0,\ell\pi),\; t>0,\\
 \ds\frac{\partial v}{\partial t}=d_2 v_{xx}+cv\left(1-\ds\frac{v}{u}\right), & x\in (0,\ell\pi),\; t>0,\\
  u_x(0,t)=
  v_x(0,t)=0,\;u_x(\ell\pi,t)=
  v_x(\ell\pi,t)=0,&
 t>0,\\
 u(x,0)=u_0(x)>0,\;\;v(x,0)=v_0(x)>0.
\end{cases}
\end{equation}
Here parameters $d_1$, $d_2$, $\beta$, $b$, $c$ and $\ell$ are all positive.
Denote
\begin{equation}\label{x}
X:=\{(u,v)^T:u,v\in H^2(0,\ell\pi),\;u_x|_{x=0,\ell\pi}=v_x|_{x=0,\ell\pi}=0\},
\end{equation}
and we can adopt the framework of \cite{hassard1981theory} to investigate the Hopf bifurcation of
model \eqref{nonlocal3} by using $b$ as a bifurcation parameter. Actually, the bifurcation parameter we choose is equivalent to parameter $b$.

The rest of the paper is organized as follows. In Section 2, we show the existence of Hopf bifurcation for model \eqref{nonlocal3}. In Section 3, we investigate the stability and direction of bifurcating periodic solutions via the center
manifold theorem and normal form theory\cite{hassard1981theory}. Finally, some numerical
simulations and spatially nonhomogeneous patterns are presented in Section 4. Throughout the paper, we denote by $\mathbb{N}$ the set of all positive integers, and $\mathbb{N}_0=\mathbb{N}\cup\{0\}$.

\section{Stability and Hopf bifurcation}
In this section, we study the stability of the constant positive equilibrium and associated Hopf bifurcation for model \eqref{nonlocal3}. Obviously, model
\eqref{nonlocal3} always has a unique constant positive equilibrium, denoted by $(\la,\la)$, where $\la\in(0,1/\beta)$
satisfies $(1-\beta\la)(1+\la)=b\la$. Therefore, $\la$
is equivalent to parameter $b$ and strictly decreasing with respect to $b$.
Linearizing system \eqref{nonlocal3} at positive equilibrium $(\la,\la)$, we obtain
\begin{equation}\label{5.2line}
\begin{cases}
  \ds\frac{\partial w}{\partial t}=d_1 w_{xx}-\ds\f{\beta\la}{\ell\pi}\int_0^{\ell \pi}w(y,t)dy+\ds\f{\la(1-\beta\la)}{1+\la}w-(1-\beta\la) z, & x\in (0,\ell\pi),\; t>0,\\
 \ds\frac{\partial z}{\partial t}-d_2 z_{xx}=cw-cz,& x\in (0,\ell\pi),\;t>0,\\
 w_x(0,t)=
  z_x(0,t)=0,\;w_x(\ell\pi,t)=
  z_x(\ell\pi,t)=0,&
 t>0.\\
\end{cases}
\end{equation}
By a direct computation, we obtain the sequence of the characteristic equations with respect to $(\la,\la)$ as follows:
\begin{equation}\label{cha}
\mu^{2}-  T_n(\la)\mu+D_n(\la)=0,~~~n\in\mathbb N_0,
\end{equation}
where
\begin{equation}\label{TD0}
T_0(\la)=-c+\ds\f{\la(1-\beta-2\beta\la)}{1+\la},\;\;
D_0(\la)=\beta c\la+\ds\f{c(1-\beta\la)}{1+
\la},
\end{equation}
 and for $n\in\mathbb N$,
\begin{equation}\label{TDn}
\begin{split}
T_n(\la)=&-c+\ds\f{\la(1-\beta\la)}{1+\la}-\ds\f{(d_1+d_2)n^2}{\ell^2},\\
D_n(\la)=&\ds\f{c(1-\beta\la)}{1+
\la}+\left[d_1c-\ds\f{d_2\la(1-\beta\la)}{1+\la}\right]\ds\f{n^2}{\ell^2}+\ds\f{d_1d_2n^4}{\ell^4}.
\end{split}
\end{equation}

In the following, we use parameter $\la$ as a bifurcation parameter to study the stability of $(\la,\la)$ and the associated Hopf bifurcation for model \eqref{nonlocal3}. Note that $\la\in(0,1/\beta)$, and equilibrium $(\la,\la)$ is locally asymptotically stable, if $T_n(\la)<0$ and $D_n(\la)>0$ for each $n\in \N_0$.
It follows from \cite{Yi} that Hopf bifurcation value $\la_0$ satisfies the following condition:

\noindent $(\mathbf{H}_1)$: There exists $n\in \mathbb{N}_0$ such
that
\begin{equation}\label{con1}
    T_n(\la_0)=0, \;\;
D_n(\la_0)>0, \;\; \text{ and } \;\; T_j(\la_0)\ne 0, \;\;
D_j(\la_0)\ne 0 \;\; \text{ for $j\ne n$},
\end{equation}
and the unique pair of complex eigenvalues $\alpha(\la)\pm i \omega(\la)$ near $\la_0$ satisfy
\begin{equation}\label{con2}
    \alpha'(\la_0)\ne 0.
\end{equation}
Denote
$$D(\la,p)=\ds\f{c(1-\beta\la)}{1+
\la}+\left[d_1c-\ds\f{d_2\la(1-\beta\la)}{1+\la}\right]p+d_1d_2p^2,$$
and then $D_n(\la)=D\left(\la,\ds\f{n^2}{\ell^2}\right)$ for $n\ge1$, where $D_n(\la)$ is defined as in Eq. \eqref{TDn}.
The following result gives a necessary and sufficient condition for $D(\la,p)>0$.
\begin{lemma}\label{dp}
$D(\la,p)>0$ for all $p\ge0$ if and only if
\begin{equation}\label{condp}
\ds\f{d_1}{d_2}>\ds\f{1-\beta\la}{c}\left(1-\sqrt{\ds\f{1}{\la+1}}\right)^2,
\end{equation}
where $\la\in(0,1/\beta)$, and $(\la,\la)$ is the unique constant positive equilibrium of model \eqref{nonlocal3}.
\end{lemma}
\begin{proof}
Obviously, $D(\la,p)>0$ for all $p\ge0$ if and only if one of the following two conditions is satisfied
\begin{enumerate}
\item [(1)] $d_1c-\ds\f{d_2\la(1-\beta\la)}{1+\la}\ge0$, or
\item [(2)] $d_1c-\ds\f{d_2\la(1-\beta\la)}{1+\la}<0$, and $\left[d_1c-\ds\f{d_2\la(1-\beta\la)}{1+\la}\right]^2-4d_1d_2\ds\frac{c(1-\beta\la)}{1+\la}<0$.
\end{enumerate}
Therefore, $D(\la,p)>0$ for all $p\ge0$ if and only if $$d_1c-\ds\f{d_2\la(1-\beta\la)}{1+\la}>-2\sqrt{\ds\f{d_1d_2c(1-\beta\la)}{1+\la}},$$
which is equivalent to Eq. \eqref{condp}. This completes the proof.
\end{proof}
Denote
\begin{equation}\label{cla}
p_1(\la)=\ds\f{1-\beta\la}{c}\left(1-\sqrt{\ds\f{1}{\la+1}}\right)^2,\;p_2(\la)=\ds\f{\la(1-\beta\la)}{1+\la},\;\text{and}\;
p_3(\la)=\ds\f{\la(1-\beta-2\beta\la)}{1+\la}.
\end{equation}
We can easily obtain the following properties on $p_i(\la)$ for $i=1,2,3$, and here we omit the proof.
\begin{lemma}\label{5.1l}
\begin{enumerate}
\item [(I)] There exists
$\la_1\in (0,1/\beta)$, which is the unique positive root of
$$\beta\left(1-\sqrt{\f{1}{\la+1}}\right)=(1-\beta\la)(\la+1)^{-\f{3}{2}},$$
such that $p_1'(\la_1)=0$, $p_1'(\la)<0$ for $\la\in(\la_1,1/\beta)$, $p_1'(\la)>0$ for $\la\in(0,\la_1)$, and
$\max_{\la\in[0,1/\beta]}p_1(\la)=p_1(\la_1)$.
\item [(II)] There exists
$\la_2:=\sqrt{\ds\f{\beta+1}{\beta}}-1$ such that $p_2'(\la_2)=0$, $p_2'(\la)<0$ for $\la\in(\la_2,1/\beta)$, $p_2'(\la)>0$ for $\la\in(0,\la_2)$, and
$\max_{\la\in[0,1/\beta]}p_2(\la)=p_2(\la_2)$.
\item [(III)] Assume that $\beta\le1$. There exists $\la_3:=\sqrt{\ds\f{\beta+1}{2\beta}}-1$ such that $p_3'(\la_3)=0$, $p_3'(\la)<0$ for $\la\in(\la_3,1/\beta)$, $p_3'(\la)>0$ for $\la\in(0,\la_3)$, and
$\max_{\la\in[0,1/\beta]}p_3(\la)=p_3(\la_3)$.
\end{enumerate}
\end{lemma}
From Lemmas \ref{dp} and \ref{5.1l}, we have:
\begin{theorem}\label{Dsign}
Assume that $d_1/d_2>p_1(\la_1)$, where $p_1(\la)$ and $\la_1$ are defined as in Eq. \eqref{cla} and Lemma \ref{5.1l} respectively. Then $D(\la,p)>0$ for any $\la\in(0,1/\beta)$ and $p\ge0$, and consequently, $D_n(\la)>0$ for each $n\in\mathbb{N}_0$ and any $\la\in(0,1/\beta)$.
\end{theorem}


Now, we show the occurrence of Hopf bifurcation for $\beta\ge1$.

\begin{theorem}\label{Hopfm}
Suppose that $d_1/d_2>p_1(\la_1)$, $\beta\ge1$,
and define \begin{equation}\label{elln}
\ell_n=n\sqrt{\ds\f{d_1+d_2}{p_2(\la_2)-c}}\; \text{  if  } \;p_2(\la_2)-c>0,
\end{equation}
where $p_i(\la), \la_i (i=1,2)$ are defined as in Eq. \eqref{cla} and Lemma \ref{5.1l} respectively. Then the following two statements are true.
\begin{enumerate}
\item [(i)] If $c\ge p_2(\la_2)$, or $c< p_2(\la_2)$ but $\ell \in (0,\ell_1)$, then $(\la,\la)$ is locally asymptotically stable for $\la\in(0,1/\beta)$.
\item [(ii)] If $c< p_2(\la_2)$ and $\ell \in (\ell_1 ,\infty)$, then there exist two points $\la_{1,+}^H$ and $\la_{1,-}^H$ satisfying
    \begin{equation}\label{Hopp}
    T_1(\la_{1,-}^H)=T_1(\la_{1,+}^H)=0\;\;\text{and}\;\;0<\la_{1,-}^H<\la_2<\la_{1,+}^H<1/\beta,
    \end{equation}
    where $T_1(\la)$ is defined as in Eq. \eqref{TDn}, such that $(\la,\la)$ is locally asymptotically stable for $\la\in\left(\la_{1,+}^H,1/\beta\right)\cup\left(0,\la_{1,-}^H\right)$ and unstable for $\la\in\left(\la_{1,-}^H,\la_{1,+}^H\right)$. Moreover, system \eqref{nonlocal3} undergoes Hopf bifurcation at $(\la,\la)$ when $\la=\la_{1,+}^H$ or $\la=\la_{1,-}^H$, and the bifurcating periodic solutions near $\la_{1,+}^H$ or $\la_{1,-}^H$ are spatially nonhomogeneous.

\end{enumerate}

\end{theorem}

\begin{proof}
It follows from $\beta\ge1$ that $T_0(\la)<0$ and $D_0(\la)>0$ for any $\la\in (0,1/\beta)$. Since
$d_1/d_2>p_1(\la_1)$, we see that $D_n(\la)>0$ for any $\la\in(0,1/\beta)$ and each $n\in\mathbb N_0$. Note that if $c\ge p_2(\la_2)$, or $c< p_2(\la_2)$ but $\ell \in (0,\ell_1)$, then $T_n(\la)<0$ for any $\la\in(0,1/\beta)$ and each $n\in\mathbb{N}$. Therefore, $(\la,\la)$ is locally asymptotically stable for $\la\in(0,1/\beta)$.

 When $\ell \in (\ell_1 ,\infty)$, there exist two points $\la_{1,+}^H$ and $\la_{1,-}^H$ such that
  $T_1\left(\la_{1,\pm}^H\right)=0$, $T_1(\la)>0$ for $\la\in\left(\la_{1,-}^H,\la_{1,+}^H\right)$, and $T_1(\la)<0$ for $\la\in\left(0,\la_{1,-}^H\right)\cup\left(\la_{1,+}^H,1/\beta\right)$.
  Therefore, $(\la,\la)$ is locally asymptotically stable for $\la\in\left(\la_{1,+}^H,1/\beta\right)\cup\left(0,\la_{1,-}^H\right)$ and unstable for $\la\in\left(\la_{1,-}^H,\la_{1,+}^H\right)$.
 When $\la$ is near $\la_{1,+}^H$ (respectively, $\la_{1,-}^H$), the unique pair of eigenvalues $\alpha(\la)\pm i\omega(\la)$ satisfy $\alpha(\la)=T_1(\la)/2$ and $\alpha'\left(\la_{1,+}^H\right)=p'_2\left(\la_{1,+}^H\right)/2$ (respectively, $\alpha'\left(\la_{1,-}^H\right)=p'_2\left(\la_{1,-}^H\right)/2$). From Lemma \ref{5.1l}, we see that
$ \alpha'\left(\la_{1,+}^H\right)<0$ and $ \alpha'\left(\la_{1,-}^H\right)>0$.
Therefore, $\la_{1,+}^H$ and $\la_{1,-}^H$ are both Hopf bifurcation points, and the bifurcating periodic solutions are spatially nonhomogeneous. This completes the proof.
\end{proof}
The case that $\beta<1$ is more complex. We first consider the case that $\beta<1$ and $c>p_3(\la_3)$. In this case, $T_0(\la)<0$ for any $\la\in(0,1/\beta)$.
Therefore, we have the similar results as in the case of $\beta\ge1$, and here we omit the proof.
\begin{theorem}\label{Hopfm2}
Suppose that $d_1/d_2>p_1(\la_1)$, $\beta<1$ and $c> p_3(\la_3)$,
where $p_i(\la), \la_i (i=1,3)$ are defined as in Eq. \eqref{cla} and Lemma \ref{5.1l} respectively. Then the following two statements are true.
\begin{enumerate}
\item [(i)] If $c\ge p_2(\la_2)$, or $p_3(\la_3)<c< p_2(\la_2)$ but $\ell \in (0,\ell_1)$, where $p_2(\la), \la_2$ are defined as in Eq. \eqref{cla} and Lemma \ref{5.1l} respectively, and $\ell_1$ is defined as in Eq. \eqref{elln}, then $(\la,\la)$ is locally asymptotically stable for $\la\in(0,1/\beta)$.
\item [(ii)] If $p_3(\la_3)<c< p_2(\la_2)$ and $\ell \in (\ell_1 ,\infty)$, then $(\la,\la)$ is locally asymptotically stable for $\la\in\left(\la_{1,+}^H,1/\beta\right)\cup\left(0,\la_{1,-}^H\right)$ and unstable for $\la\in\left(\la_{1,-}^H,\la_{1,+}^H\right)$, where $\la_{1,-}^H$ and $\la_{1,+}^H$ are defined as in Eq. \eqref{Hopp}. Moreover, system \eqref{nonlocal3} undergoes Hopf bifurcation at $(\la,\la)$ when $\la=\la_{1,+}^H$ or $\la=\la_{1,-}^H$, and the bifurcating periodic solutions near $\la_{1,+}^H$ or $\la_{1,-}^H$ are spatially nonhomogeneous.

\end{enumerate}
\end{theorem}
Then we consider the case that $\beta<1$ and $c<p_3(\la_3)$. In this case, there exists two points $\la_{0,-}^H$ and $\la_{0,+}^H$
satisfying
\begin{equation}\label{Hopf0}
T_0(\la_{0,-}^H)=T_0(\la_{0,+}^H)\;\;\text{and}\;\;\la_{0,-}^H<\la_3<\la_{0,+}^H.
\end{equation}
These two points are possible Hopf bifurcation points, and the bifurcating periodic solutions are spatially homogeneous.
\begin{theorem}\label{Hopfm3}
Suppose that $d_1/d_2>p_1(\la_1)$, $\beta<1$, and $c<p_3(\la_3)$, where $p_i, \la_i (i=1,3)$ are defined as in Eq. \eqref{cla} and Lemma \ref{5.1l} respectively.
Define
\begin{equation}\label{tildelln}
\tilde\ell^+_n=n\sqrt{\ds\f{d_1+d_2}{\beta\la_{0,+}^H}} \;\;\text{and}\;\;\tilde\ell^-_n=n\sqrt{\ds\f{d_1+d_2}{\beta\la_{0,-}^H}},
\end{equation}
 where $\la_{0,-}^H$ and $\la_{0,+}^H$ are defined as in Eq. \eqref{Hopf0}.
\begin{enumerate}
\item [(i)]For $\la_{0,+}^H>\la_2$, where $\la_2$ is defined as in Lemma \ref{5.1l}, the following results hold.
\begin{enumerate}
\item [($i_1$)] If $\ell\in(0,\tilde\ell_1^+)$, then $(\la,\la)$ is locally asymptotically stable for $\la\in(0,\la_{0,-}^H)\cup(\la_{0,+}^H,1/\beta)$ and unstable for $\la\in(\la_{0,-}^H,\la_{0,+}^H)$. Moreover, system \eqref{nonlocal3} undergoes Hopf bifurcation at $(\la,\la)$ when $\la=\la_{0,+}^H$ or $\la=\la_{0,-}^H$, and the bifurcating periodic solutions near $\la_{0,+}^H$ or $\la_{0,-}^H$ are spatially homogeneous.
\item [($i_2$)] If $\ell\in(\tilde\ell_1^+,\tilde\ell_1^-)$, then $\la_{0,-}^H<\la_{1,-}^H<\la_2<\la_{0,+}^H<\la_{1,+}^H$, where $\la_{1,-}^H$ and $\la_{1,+}^H$ are defined as in Eq. \eqref{Hopp},
and $(\la,\la)$ is locally asymptotically stable for $\la\in(0,\la_{0,-}^H)\cup(\la_{1,+}^H,1/\beta)$ and unstable for $\la\in(\la_{0,-}^H,\la_{1,+}^H)$. Moreover, system \eqref{nonlocal3} undergoes Hopf bifurcation at $(\la,\la)$ when $\la=\la_{0,-}^H$ or $\la=\la_{1,-}^H$, and the bifurcating periodic solutions near $\la_{0,-}^H$ (respectively, $\la_{1,+}^H$) are spatially homogeneous (respectively, spatially nonhomogeneous).
\item [($i_3$)] If $\ell\in (\tilde\ell_1^-,\infty) $,
then $\la_{1,-}^H<\la_{0,-}^H<\la_2<\la_{0,+}^H<\la_{1,+}^H$, and $(\la,\la)$ is locally asymptotically stable for $\la\in\left(\la_{1,+}^H,1/\beta\right)\cup\left(0,\la_{1,-}^H\right)$ and unstable for $\la\in\left(\la_{1,-}^H,\la_{1,+}^H\right)$. Moreover, system \eqref{nonlocal3} undergoes Hopf bifurcation at $(\la,\la)$ when $\la=\la_{1,+}^H$ or $\la=\la_{1,-}^H$, and the bifurcating periodic solutions near $\la_{1,+}^H$ or $\la_{1,-}^H$ are spatially nonhomogeneous.
\end{enumerate}
\item [(ii)] For $\la_{0,+}^H<\la_2$, the following results hold.
\begin{enumerate}
\item [($ii_1$)] If $\ell\in (0,\ell_1)$, where $\ell_1$ is defined as in Eq. \eqref{elln}, then $(\la,\la)$ is locally asymptotically stable for $\la\in(0,\la_{0,-}^H)\cup(\la_{0,+}^H,1/\beta)$ and unstable for $\la\in(\la_{0,-}^H,\la_{0,+}^H)$. Moreover, system \eqref{nonlocal3} undergoes Hopf bifurcation at $(\la,\la)$ when $\la=\la_{0,+}^H$ or $\la=\la_{0,-}^H$, and the bifurcating periodic solutions near $\la_{0,+}^H$ or $\la_{0,-}^H$ are spatially homogeneous.
\item [($ii_2$)] If $\ell\in(\ell_1,\tilde\ell_1^+)$, then $\la_{0,-}^H<\la_{0,+}^H<\la_{1,-}^H<\la_2<\la_{1,+}^H$,
and $(\la,\la)$ is locally asymptotically stable for $\la\in(0,\la_{0,-}^H)\cup(\la_{0,+}^H,\la_{1,-}^H)\cup(\la_{1,+}^H,1/\beta)$ and unstable for $\la\in(\la_{0,-}^H,\la_{0,+}^H)\cup(\la_{1,-}^H,\la_{1,+}^H)$. Moreover, system \eqref{nonlocal3} undergoes Hopf bifurcation at $(\la,\la)$ when $\la=\la_{0,-}^H$ or $\la=\la_{1,+}^H$, and the bifurcating periodic solutions near $\la_{0,-}^H$ (respectively, $\la_{1,+}^H$) are spatially homogeneous (respectively, spatially nonhomogeneous).
\item [($ii_3$)] If $\ell\in(\tilde\ell_1^+,\tilde\ell_1^-)$, then $\la_{0,-}^H<\la_{1,-}^H<\la_{0,+}^H<\la_2<\la_{1,+}^H$,
and $(\la,\la)$ is locally asymptotically stable for $\la\in(0,\la_{0,-}^H)\cup(\la_{1,+}^H,1/\beta)$ and unstable for $\la\in(\la_{0,-}^H,\la_{1,+}^H)$. Moreover, system \eqref{nonlocal3} undergoes Hopf bifurcation at $(\la,\la)$ when $\la=\la_{0,-}^H$ or $\la=\la_{1,-}^H$, and the bifurcating periodic solutions near $\la_{0,-}^H$ (respectively, $\la_{1,+}^H$) are spatially homogeneous (respectively, spatially nonhomogeneous).
\item [($ii_4$)] If $\ell\in (\tilde\ell_1^-,\infty) $,
then $\la_{1,-}^H<\la_{0,-}^H<\la_{0,+}^H<\la_2<\la_{1,+}^H$, and $(\la,\la)$ is locally asymptotically stable for $\la\in\left(\la_{1,+}^H,1/\beta\right)\cup\left(0,\la_{1,-}^H\right)$ and unstable for $\la\in\left(\la_{1,-}^H,\la_{1,+}^H\right)$. Moreover, system \eqref{nonlocal3} undergoes Hopf bifurcation at $(\la,\la)$ when $\la=\la_{1,+}^H$ or $\la=\la_{1,-}^H$, and the bifurcating periodic solutions near $\la_{1,+}^H$ or $\la_{1,-}^H$ are spatially nonhomogeneous.
\end{enumerate}
\end{enumerate}
\end{theorem}

\begin{proof}
We only prove part ($ii$), and part ($i$) can be proved similarly. A direct computation yields
\begin{equation*}
\begin{split}
&T_1(\la_{0,\pm}^H)=\beta\la_{0,\pm}^H-\ds\f{(d_1+d_2)n^2}{\ell^2},\\
&T_1(\la_2)=\max_{\la\in(0,1/\beta)}T_1(\la)=p_2(\la_2)-c-\ds\f{(d_1+d_2)n^2}{\ell^2},
\end{split}
\end{equation*}
and $\ell_1<\tilde\ell_1^+<\tilde\ell_1^-$.
Therefore,
\begin{enumerate}
\item [(1)]if $\ell\in (0,\ell_1)$, then $T_1(\la)<0$ for any $\la\in(0,1/\beta)$;
\item [(2)] if $\ell>\ell_1$, then $T_1(\la)=0$ has two positive roots $\la_{1,+}^H$ and $\la_{1,-}^H$ such that
$T_1(\la)>0$ for $\la\in (\la_{1,-}^H,\la_{1,+}^H)$ and $T_1(\la)<0$ for $\la\in(0,\la_{1,-}^H)\cup (\la_{1,+}^H,1/\beta)$. Moreover, $\la_{0,-}^H<\la_{0,+}^H<\la_{1,-}^H<\la_2<\la_{1,+}^H$ for $\ell\in(\ell_1, \tilde\ell_1^+)$, $\la_{0,-}^H<\la_{1,-}^H<\la_{0,+}^H<\la_2<\la_{1,+}^H$ for $\ell\in(\tilde\ell_1^+,\tilde\ell_1^-)$, and $\la_{1,-}^H<\la_{0,-}^H<\la_{0,+}^H<\la_2<\la_{1,+}^H$ for $\ell\in (\tilde\ell_1^-,\infty)$.
\end{enumerate}
Note that when $\la$ is near $\la_{i,\pm}^H$, the unique pair of eigenvalues $\alpha(\la)\pm i\omega(\la)$ satisfy $\alpha(\la)=T_i(\la)/2$, and consequently, $\alpha'\left(\la_{1,\pm}^H\right)=p'_2\left(\la_{1,\pm}^H\right)/2$, and $\alpha'\left(\la_{0,\pm}^H\right)=p'_3\left(\la_{0,\pm}^H\right)/2$. Therefore
$ \alpha'\left(\la_{i,+}^H\right)<0$ and $ \alpha'\left(\la_{i,-}^H\right)>0$ for $i=0,1$, and
part $(ii)$ is proved.

\end{proof}
\begin{remark}
To investigate the effect of the non-locality, we revisit the classical Holling-Tanner predator-prey model without nonlocal effect,
\begin{equation}\label{classi}
\begin{cases}
  \ds\frac{\partial u}{\partial t}=d_1 u_{xx}+u\left(1-\beta u\right)-\ds\frac{buv}{u+1}, & x\in (0,\ell\pi),\; t>0,\\
 \ds\frac{\partial v}{\partial t}=d_2 v_{xx}+cv\left(1-\ds\frac{v}{u}\right), & x\in (0,\ell\pi),\; t>0,\\
  u_x(0,t)=
  v_x(0,t)=0,\;u_x(\ell\pi,t)=
  v_x(\ell\pi,t)=0,&
 t>0,\\
 u(x,0)=u_0(x)>0,\;\;v(x,0)=v_0(x)>0.
\end{cases}
\end{equation}
For the case of $\beta>1$, a direct calculation implies that there exist no Hopf bifurcation points for model \eqref{classi}, which satisfy assumption $(\mathbf{H}_1)$, (see also
\cite{LiJiang}). However, it follow form Theorem \ref{Hopfm} that under certain conditions Hopf bifurcation can occur for model \eqref{nonlocal3} with nonlocal effect, and the bifurcating periodic solutions are spatially nonhomogeneous. Therefore,
Hopf bifurcation is more likely to occur with the nonlocal competition of prey. Similar results can be obtained for the case of $\beta<1$.
\end{remark}

\begin{remark}\label{re}
It follows from Theorems \ref{Hopfm}-\ref{Hopfm3} that, if $d_1/d_2>p_1(\la)$ and $c<p_2(\la)$, then, for sufficiently large $\ell$, $(\la,\la)$ is locally asymptotically stable for $\la\in\left(\la_{1,+}^H,1/\beta\right)\cup\left(0,\la_{1,-}^H\right)$ and unstable for $\la\in\left(\la_{1,-}^H,\la_{1,+}^H\right)$. Moreover, system \eqref{nonlocal3} undergoes Hopf bifurcation at $(\la,\la)$ when $\la=\la_{1,+}^H$ or $\la=\la_{1,-}^H$, and the bifurcating periodic solutions near $\la_{1,+}^H$ or $\la_{1,-}^H$ are spatially nonhomogeneous. Therefore, spatial nonhomogeneous periodic solutions are more likely to occur with
large spatial scale $\ell$.
\end{remark}

\section{The direction and stability of Hopf bifurcation}
From discussions in Section 2, we see that, under certain conditions, model \eqref{nonlocal3} undergoes Hopf bifurcation at $(\la,\la)$ when $\la=\la_{1,+}^H$ or $\la=\la_{1,-}^H$, and the bifurcating periodic solutions are spatially nonhomogeneous.
In this section, we will adopt the framework of Hassard et al. (see Chapter 5 in \cite{hassard1981theory}) to investigate the direction of Hopf bifurcation and the stability of the bifurcating spatially nonhomogeneous periodic solutions.

Setting $\tilde U(x,t)=(\tilde u(x,t),\tilde v(x,t))^T=(u(x,t)-\la,v(x,t)-\la)^T$, and dropping the tilde sign,
system \eqref{nonlocal3} can be transformed as follows:
\begin{equation}\label{ab}
\ds\frac{d U(t)}{dt}=L(\la)U+F(\la,U),
\end{equation}
where
\begin{equation*}
\begin{split}
&L(\la) U=\left(\begin{array}{c}d_1u_{xx}-\f{\beta\la}{\ell\pi}\int_0^{\ell\pi} u(y,t)dy+\frac{\la(1-\beta\la)}{1+\la} u-(1-\beta\la) v\\
d_2  v_{xx}+cu-cv
\end{array}\right),\\
&F(\la, U)\\
=&\left(\begin{array}{c}\frac{1-\beta\la}{1+\la} u+\la(1-\beta\la)-\f{\beta u}{\ell\pi}\int_0^{\ell\pi} u(y,t)dy +(1-\beta\la) v-\frac{b( u+\la)(v+\la)}{ u+\la+1}\\
-cu+c v+c( v+\la)\left(1-\f{ v+\la}{u+\la}\right)
\end{array}\right),
\end{split}
\end{equation*}
for $U=(u,v)^T\in X_{\mathbb C}$.

Letting $\langle \cdot,\cdot\rangle$ be the complex-valued $L^2$ inner product for space $X_{\mathbb C}$, defined by
\begin{equation}\label{in}
\langle U, V\rangle=\int_0^{\ell\pi}\left(\overline u_1 v_1+\overline u_2 v_2\right)dx,
\end{equation}
where $U=(u_1,u_2)^T$ and $V=(v_1,v_2)^T$. Denote the adjoint operator of $L(\la)$ by $L^*(\la)$, which satisfies
$\langle  U, L(\la)V\rangle=\langle L^*(\la)U,  V\rangle$ for any $U,V\in X_{\mathbb C}$. A direct calculation leads to
\begin{equation}\label{L0s}
L^*(\la)\tilde U=\left(\begin{array}{c}d_1\tilde u_{xx}-\f{\beta\la}{\ell\pi}\int_0^{\ell\pi} \tilde u(y,t)dy+\frac{\la(1-\beta\la)}{1+\la} \tilde u+c \tilde v\\
d_2  \tilde v_{xx}-(1-\beta\la)\tilde u-c\tilde v
\end{array}\right),\\
\end{equation}
for $\tilde U=( \tilde u,\tilde v)^T\in X_{\mathbb C}$.
For simplicity of notations, we denote
\begin{equation}\label{la0}
\la_{*}=\la_{1,+}^H\;\text{or}\;\la_{1,-}^H.
\end{equation}
Since the formulas of Hopf bifurcation are all relative to $\la=\la_{*}$ only, we set $\la=\la_{*}$ in
Eq. \eqref{ab} and obtain
\begin{equation}\label{ab0}
\ds\frac{d U(t)}{dt}=L(\la_{*})U+F(\la_{*},U).
\end{equation}
It follows from Section 2 that $L(\la_{*})$ and $L^*(\la_{*})$ has only one pair of purely imaginary eigenvalue $\pm i\omega_{*}$, where
\begin{equation}\label{omega}
\omega_{*}=\sqrt{D_1(\la_{*})},
\end{equation}
and other eigenvalues are all negative. Let
$q$ and  $q^*$ satisfy
$$L(\la_{*})q=iw_{*}q,\;\;L^*(\la_{*})q^*=-i\omega_{*}q^*\;\; \text{and}\;\;\langle q^*,q\rangle=1.$$
Since $L^*(\la_{*})$ is the adjoint operator of $L(\la_{*})$, we see that
\begin{equation*}
\begin{split}
&\langle L^*(\la_*)q^*,\overline q\rangle=\langle -i\omega_*q^*,\overline q\rangle =i\omega_*\langle q^*,\overline q\rangle\\
=&\langle q^*,L(\la_*)\overline q\rangle=\langle q^*,-i\omega_*\overline q\rangle=-i\omega_*\langle q^*,\overline q\rangle,
\end{split}
\end{equation*}
which leads to $\langle q^*,\overline q\rangle=0$.
In fact, we can choose
\begin{equation}\label{qqs}
\begin{split}
q=(q_1,q_2)^T\cos\f{x}{\ell},\;\;q^*=\f{2}{\ell\pi\overline D}(q_1^*,q_2^*)^T\cos\f{x}{\ell},
\end{split}
\end{equation}
where
\begin{equation}\label{D}
\begin{split}
&q_1=q_1^*=1,\;\;q_2=\f{c}{i\omega_{*}+\f{d_2}{\ell^2}+c},\;\;q_2^*=\f{1-\beta\la_*}{i\omega_{*}-\f{d_2}{\ell^2}-c},\\
&D=q_1\overline q_1^*+q_2\overline {q_2^*}=1-\f{c(1-\beta\la)}{\left(i\omega_{*}+\f{d_2}{\ell^2}+c\right)^2}.
\end{split}
\end{equation}
Decompose $X_{\mathbb{C}}=X^c\oplus X^s$, where $X^c=\{zq+\overline {z}{\overline q}:z\in\mathbb C\}$ and
$X^s=\{u\in X_{\mathbb C}:\langle q^*, u\rangle=\langle \overline {q^*}, u\rangle=0\}$.
Here $X^c$ and $X^s$ are the center subspace and stable subspace of system \eqref{ab0} respectively.
Then system \eqref{ab0} in $z$ and $w$ coordinates takes the following form:
\begin{equation}\label{zw}
\begin{cases}
\ds\f{dz}{dt}=\ds\f{d}{dt}\langle q^*,U(t)\rangle
=\langle q^*,
L(\la_{*})U\rangle+\langle q^*, F(\la_{*},U)\rangle
=i\omega_{*}z+\langle q^*, F(\la_{*},U)\rangle,\\
\ds\f{dw}{dt}=L(\la_{*})w+F(\la_{*},U)-\langle q^*, F(\la_{*},U)\rangle q-\langle \overline{q}^*, F(\la_{*},U)\rangle \overline q,
\end{cases}
\end{equation}
where $U=z(t)q+\overline z(t) \overline q+w(t)$.
It follows from \cite{hassard1981theory} that system \eqref{zw} posses a center manifold in the following form:
\begin{equation}\label{center}w(z,\overline
z)=w_{20}\frac{z^{2}}{2}+w_{11}z\overline
z+w_{02}\frac{\overline z^{2}}{2}+O(|z|^3)\end{equation}
with $w_{ij}=(w_{ij}^{(1)},w_{ij}^{(2)})^T$ in $X^s$ for $i+j=2$.
Therefore, the flow of system \eqref{zw} on the center
manifold can be written as:
\begin{equation*}U(t)=z(t)q+\overline z(t) \overline q+w(z(t),\overline z(t)),\end{equation*}
where $z(t)$ satisfies
\begin{equation}\label{z(t)}
\dot{z}(t)
=i\omega_{*}z(t)+g(z,\overline z).
\end{equation}
Here
\begin{equation}\label{gz}
g(z,\overline z)=\langle q^*, F\left(\la_{*},z(t)q+\overline z(t)\overline q+w(z(t),\overline z(t))\right)\rangle\\
=\sum_{2\le i+j\le 3}\f{g_{ij}}{i!j!}z^i\overline z^j+O(|z|^4).
\end{equation}
Note that $\int_0^{\ell\pi}\cos^3\f{x}{\ell}dx=0$ and
\begin{equation}\label{F}
\begin{split}
&F(\la_{*}, U)\\
=&\left(\begin{array}{c}\f{1-\beta\la_{*}}{(\la_{*}+1)^2}u^2-\f{1-\beta\la_{*}}{\la_{*}(\la_{*}+1)}uv-\f{1-\beta\la_{*}}{(1+\la_{*})^3}u^3
+\f{1-\beta\la_{*}}{\la_{*}(\la_{*}+1)^2}u^2v-\f{\beta u}{\ell\pi}\int_0^{\ell\pi}u(y,t)dy\\
-\f{c}{\la_*}u^2-\f{c}{\la_*}v^2+\f{2c}{\la_*}uv+\f{c}{\la_*^2}u^3-\f{2c}{\la_*^2}u^2v+\f{c}{\la_*^2}uv^2
\end{array}\right)\\
+&O(\|U\|^4_{X_{\mathbb C}}),
\end{split}
\end{equation}
where $U=(u,v)^T\in X_{\mathbb C}$ and $\|U\|_{X_{\mathbb C}}^2=\int_0^{\ell\pi}u^2dx+\int_0^{\ell\pi}v^2dx+\int_0^{\ell\pi}u_x^2dx+\int_0^{\ell\pi}v_x^2dx$.
An easy calculation implies that
\begin{equation}\label{g}
\begin{split}
&g_{20}=g_{11}=g_{02}=0,\\
&\ds\f{\ell\pi D}{2}g_{21}\\=&\f{2(1-\beta\la_{*})}{(\la_{*}+1)^2}\int_0^{\ell\pi}\cos^2\f{x}{\ell}\left(2w_{11}^{(1)}+w_{20}^{(1)}\right)dx\\
-&\ds\f{2(1-\beta\la_{*})}{\la_{*}(1+\la_{*})}\int_0^{\ell\pi}\cos^2\f{x}{\ell}\left(w_{11}^{(2)}+
\ds\f{w_{20}^{(2)}}{2}+\overline q_2\ds\f{w_{20}^{(1)}}{2}+q_2w_{11}^{(1)}\right)dx\\
-&\f{6(1-\beta\la_{*})}{(1+\la_{*})^3}\int_0^{\ell\pi}\cos^4\f{x}{\ell}dx
+\ds\f{2(1-\beta\la_{*})}{\la_{*}(1+\la_{*})^2}
\left(\overline q_2+2q_2\right)\int_0^{\ell\pi}\cos^4\f{x}{\ell}dx
\\
-&\ds\f{2\beta}{\ell\pi}\int_0^{\ell\pi}\cos^2\f{x}{\ell}dx\int_0^{\ell\pi}w_{11}^{(1)}dx
-\f{\beta}{\ell\pi}\int_0^{\ell\pi}\cos^2\f{x}{\ell}dx\int_0^{\ell\pi}w_{20}^{(1)}dx\\
-&\f{2c}{\la_{*}}\overline {q_2^*}(1-\overline q_2)\int_0^{\ell\pi}\cos^2\f{x}{\ell}(w_{20}^{(1)}-w_{20}^{(2)})dx\\
-&\f{4c}{\la_{*}}\overline {q_2^*}(1-q_2)\int_0^{\ell\pi}\cos^2\f{x}{\ell}(w_{11}^{(1)}-w_{11}^{(2)})dx
-\f{4c}{\la_{*}^2}\overline {q_2^*}(\overline q_2+2q_2)\int_0^{\ell\pi}\cos^4\f{x}{\ell}dx\\
+&\f{6c}{\la_{*}^2}\overline {q_2^*}\int_0^{\ell\pi}\cos^4\f{x}{\ell}dx
+\f{2c}{\la_{*}^2}\overline {q_2^*}(q_2^2+2q_2\overline q_2)\int_0^{\ell\pi}\cos^4\f{x}{\ell}dx.
\end{split}
\end{equation}
Therefore, to derive the expression of $g_{21}$, one need to compute $w_{20}$
and $w_{11}$.

It follows from system \eqref{zw} that $w(z(t),(z(t))$ satisfies
\begin{equation}\label{3.12}
\begin{split}
\dot w=&L(\la_{*})w+F(\la_{*},zq+\overline z\overline q+w(z,\overline z))\\
-&\langle q^*, F(\la_{*},zq+\overline z\overline q+w(z,\overline z))\rangle q-\langle \overline{q}^*, F(\la_{*},zq+\overline z\overline q+w(z,\overline z))\rangle \overline q\\
=&L(\la_{*})w+h_{20}\ds\frac{z^2}{2}+h_{11}z\overline
z+h_{02}\ds\frac{\overline z^2}{2}+O(|z|^3).
\end{split}
\end{equation}
Then by using the chain rule
\begin{equation*}\dot w=\ds\frac{\partial w(z,\overline z)}{\partial z}\dot z+\ds\frac{\partial w(z,\overline z)}{\partial \overline z}\dot{\overline z},
\end{equation*}
we have
\begin{equation}\label{3.13}
\begin{cases}
w_{20}=[2i\omega_{*}-L(\la_{*})]^{-1}h_{20},\\
w_{11}=-\left[L(\la_{*})\right]^{-1}h_{11}.\\
\end{cases}
\end{equation}
Since $g_{20}=g_{11}=g_{02}=0$, we obtain that
\begin{equation}\label{h}
\begin{split}
h_{20}=&\left(h_{20}^{(1)},h_{20}^{(2)}\right)^T=\left(\gamma_1,\gamma_2\right)^T\cos^2\f{x}{\ell},\\
h_{11}=&\left(h_{11}^{(1)},h_{11}^{(2)}\right)^T=\left(\gamma_3,\gamma_4\right)^T\cos^2\f{x}{\ell},\\
\end{split}
\end{equation}
where
\begin{equation}\label{gamm}
\begin{split}
&\gamma_1=\f{2(1-\beta\la_{*})}{(1+\la_{*})^2}-\f{2(1-\beta\la_{*})}{\la_{*}(1+\la_{*})}q_2,\;\;
\gamma_2=-\f{2c}{\la_{*}}(1-q_2)^2,\;\;\\
&\gamma_3=\f{2(1-\beta\la_{*})}{(1+\la_{*})^2}-\f{(1-\beta\la_{*})}{\la_{*}(1+\la_{*})}(\overline q_2+ q_2),\;\;\gamma_4=-\f{2c}{\la_{*}}(1-q_2)(1-\overline q_2).
\end{split}
\end{equation}
Consequently,
\begin{equation}\label{w}
\begin{split}
w_{20}=&\left(w_{20}^{(1)},w_{20}^{(2)}\right)^T=(a_1,a_2)^T\cos\f{2x}{\ell}+(a_3,a_4)^T,\\
w_{11}=&\left(w_{11}^{(1)},w_{11}^{(2)}\right)^T=(b_1,b_2)^T\cos\f{2x}{\ell}+(b_3,b_4)^T,\\
\end{split}
\end{equation}
where
\begin{equation}\label{abv}
\begin{split}
a_1=&\ds\f{1}{2}\ds\f{\gamma_1\left(2i\omega_*+d_2\f{4}{\ell^2}+c\right)-\gamma_2(1-\beta\la_*)}{-4\omega_*^2-2iT_2(\la_*)\omega_*+D_2(\la_*)},
\;a_3=\ds\f{1}{2}\ds\f{\gamma_1\left(2i\omega_*+c\right)-\gamma_2(1-\beta\la_*)}{-4\omega_*^2-2iT_0(\la_*)\omega_*+D_0(\la_*)},\\
a_2=&\ds\f{1}{2}\ds\f{\gamma_2\left(2i\omega_*+d_1\f{4}{\ell^2}-\f{\la_*(1-\beta\la_*)}{1+\la_*}\right)+c\gamma_1}{-4\omega_*^2-2iT_2(\la_*)\omega_*+D_2(\la_*)},\;
a_4=\ds\f{1}{2}\ds\f{\gamma_2\left(2i\omega_*-\f{\la_*(1-\beta-2\beta\la_*)}{1+\la_*}\right)+c\gamma_1}{-4\omega_*^2-2iT_0(\la_*)\omega_*+D_0(\la_*)},\\
b_1=&\ds\f{1}{2}\ds\f{\gamma_3\left(d_2\f{4}{\ell^2}+c\right)-\gamma_4(1-\beta\la_*)}{D_2(\la_*)},\;\;b_3=\ds\f{1}{2}\ds\f{\gamma_3c-\gamma_4(1-\beta\la_*)}{D_0(\la_*)},\\
b_2=&\ds\f{1}{2}\ds\f{\gamma_4\left(d_1\f{4}{\ell^2}-\f{\la_*(1-\beta\la_*)}{1+\la_*}\right)+c\gamma_3}{D_2(\la_*)},
\;\;b_4=\ds\f{1}{2}\ds\f{\gamma_4\left(-\f{\la_*(1-\beta-2\beta\la_*)}{1+\la_*}\right)+c\gamma_3}{D_0(\la_*)}.\\
\end{split}
\end{equation}
Substituting Eqs. \eqref{w} and \eqref{abv} into the last equation of \eqref{g}, we can compute the value of $g_{21}$ for given parameters.
Then we can compute the following
quantities which determine the direction and stability of
bifurcating spatially nonhomogeneous periodic solutions:
\begin{equation*}
\begin{split}
&C_1(0)=\dfrac{i}{2\omega_{*}}\left(g_{11}g_{20}-2|g_{11}|^2
-\dfrac{|g_{02}|^2}{3}\right)+\dfrac{g_{21}}{2}=\dfrac{g_{21}}{2},\\
&\mu_2=-\dfrac{{\rm Re}(C_1(0))}{{\rm Re}(\mu'(\la_{*}))},\;\;
\beta_2=2{\rm Re}(C_1(0))={\rm Re}(g_{21}).
\end{split}
\end{equation*}
Here
\begin{enumerate}
\item[(1)] $\mu_2$ determines the direction of the Hopf bifurcation. If
$\mu_2>0$ (respectively, $\mu_2<0$), then the bifurcating periodic solutions exist
in the right neighborhood of $\la_{*}$ (respectively, in the left neighborhood of $\la_{*}$);
\item[(2)] $\beta_2$ determines
the stability of bifurcating spatially nonhomogeneous periodic solutions. If $\beta_2<0$ (respectively, $\beta_2>0$), then the bifurcating
periodic solutions are orbitally asymptotically stable (respectively, unstable).
\end{enumerate}

Since the expression of $g_{21}$ is complex, we can only determine the sign of ${\rm Re}(g_{21})$ for given parameters.
From Theorem
\ref{Hopfm}-\ref{Hopfm3}, we see that
if $d_1/d_2>p_1(\la_1)$ and $c<p_2(\la_2)$, then, for sufficiently large $\ell$,
\begin{enumerate}
 \item [(1)]$(\la,\la)$ is locally asymptotically stable for $\la\in\left(\la_{1,+}^H,1/\beta\right)\cup\left(0,\la_{1,-}^H\right)$ and unstable for $\la\in\left(\la_{1,-}^H,\la_{1,+}^H\right)$.
 \item [(2)] system \eqref{nonlocal3} undergoes Hopf bifurcation at $(\la,\la)$ when $\la=\la_{1,+}^H$ or $\la=\la_{1,-}^H$, and the bifurcating periodic solutions near $\la_{1,+}^H$ or $\la_{1,-}^H$ are spatially nonhomogeneous.
\end{enumerate}
In the following, we will consider the sign of ${\rm Re}(g_{21})$ with respect to Hopf bifurcation point $\la_{1,-}^H$ or $\la_{1,+}^H$ for large $\ell$. Firstly, we show the limit of
$\la_{*}$, $\omega_{*}$ $q_2$, $q_2^*$ and $D$ as $\ell\to\infty$ for further application.
\begin{lemma}\label{ll3}
Denote
\begin{equation}\label{3.21}
\begin{split}
&\la_{*}^\infty=\lim_{\ell\to\infty}\la_{*}, \;\;\omega_{*}^\infty=\lim_{\ell\to\infty}\omega_{*}, \\
&q_2^\infty=\lim_{\ell\to\infty}q_2,\;\; q_2^{*\infty}=\lim_{\ell\to\infty}q_2^*, \;\;D^\infty=\lim_{\ell\to\infty} D,
\end{split}
\end{equation}
where $\la_{*}$, $\omega_{*}$ $q_2$, $q_2^*$ and $D$ are defined as in Eqs. \eqref{la0}, \eqref{omega}, \eqref{qqs} and \eqref{D} respectively. Then
\begin{equation}\label{inf}
\begin{split}
&c=\ds\f{\la_{*}^\infty(1-\beta\la_{*}^\infty)}{1+\la_{*}^\infty}, \;\;\left(\omega_{*}^\infty\right)^2=\ds\f{c(1-\beta\la_{*}^\infty)}{1+\la_{*}^\infty},\;\;
q_2^{*\infty}=-1-i\ds\f{\omega_{*}^\infty}{c}, \\
&q_2^\infty=\ds\f{c}{1-\beta\la_{*}^\infty}-i\ds\f{\omega_{*}^\infty}{1-\beta\la_{*}^\infty}
=\ds\f{\la_{*}^\infty}{1+\la_{*}^\infty}-i\ds\f{c}{\omega_{*}^\infty\left(1+\la_{*}^\infty\right)},\\
&D^\infty=\ds\f{2}{1+\la_{*}^\infty}\left(1+i\ds\f{c}{\omega_{*}^\infty}\right)=\f{2}{1-i\f{c}{\omega_{*}^\infty}}.
\end{split}
\end{equation}
\end{lemma}
\begin{proof}
Taking the limit of equation $T_1(\la_*)=0$ and $D_1(\la_*)=\omega_*^2$ as $\ell\to\infty$, we have
$$c=\ds\f{\la_{*}^\infty(1-\beta\la_{*}^\infty)}{1+\la_{*}^\infty}, \;\;\left(\omega_{*}^\infty\right)^2=\ds\f{c(1-\beta\la_{*}^\infty)}{1+\la_{*}^\infty},$$
which implies that
\begin{equation}\label{rela}
\la_*^\infty\left(\omega_{*}^\infty\right)^2=c^2.
\end{equation}
Note that $\langle q^*,\overline q\rangle=0$. Then
$$\left(\omega_*\right)^2+\left(c+\ds\f{d_2}{\ell^2}\right)^2=c(1-\beta\la_{*}),$$
which leads to
\begin{equation*}
\begin{split}
q_2^\infty=&\lim_{\ell\to\infty}\f{c}{i\omega_{*}+\f{d_2}{\ell^2}+c}
=\lim_{\ell\to\infty}\ds\f{c}{c(1-\beta\la_*)}\left(c+\ds\f{d_2}{\ell^2}-i\omega_*\right)\\
=&\ds\f{c}{1-\beta\la_{*}^\infty}-i\ds\f{\omega_{*}^\infty}{1-\beta\la_{*}^\infty}
=\ds\f{\la_{*}^\infty}{1+\la_{*}^\infty}-i\ds\f{c}{\omega_{*}^\infty\left(1+\la_{*}^\infty\right)},\\
q_2^{*\infty}=&\lim_{\ell\to\infty}\f{1-\beta\la_*}{i\omega_{*}-\f{d_2}{\ell^2}-c}
=-\lim_{\ell\to\infty}\ds\f{1-\beta\la_*}{c(1-\beta\la_*)}\left(c+\ds\f{d_2}{\ell^2}+i\omega_*\right)\\
=&-1-i\ds\f{\omega_{*}^\infty}{c}.\\
\end{split}
\end{equation*}
This, combined with Eq. \eqref{rela}, implies that
\begin{equation}
\begin{split}
D^\infty=&\lim_{\ell\to\infty}(q_1\overline q_1^*+q_2\overline {q_2^*})=\ds\f{2}{1+\la_{*}^\infty}
+i\ds\f{1}{1+\la_{*}^\infty}\left(\f{\la_*^\infty\omega_{*}^\infty}{c}+\f{c}{\omega_{*}^\infty}\right)\\
=&\ds\f{2}{1+\la_{*}^\infty}\left(1+i\f{c}{\omega_{*}^\infty}\right)=\f{2}{1-i\f{c}{\omega_{*}^\infty}}.
\end{split}
\end{equation}
\end{proof}
By virtue of Lemma \ref{ll3}, we obtain the limits of $\gamma_i$ as $\ell\to\infty$ for $i=1,2,3,4$.
\begin{lemma}\label{hinf}
Denote
\begin{equation}
\gamma_i^\infty=\lim_{\ell\to\infty}\gamma_i\;\;\text{for}\;\;i=1,2,3,4,
\end{equation}
where $\gamma_i\;(i=1,2,3,4)$ are defined as in Eq. \eqref{gamm}.
Then
\begin{equation}\label{gainf}
\begin{split}
&\gamma_1^\infty=i\ds\f{2\omega_*^\infty}{\la_{*}^\infty\left(1+\la_{*}^\infty\right)},\;\;
\gamma_2^\infty=-\ds\f{2c}{\la_{*}^\infty\left(1+\la_{*}^\infty\right)^2}\left(1-\la_{*}^\infty+2i\ds\f{c}{\omega_*^\infty}\right),\;\;\\
&\gamma_3^\infty=0,\;\;\gamma_4^\infty=-\ds\f{2c}{\la_{*}^\infty\left(1+\la_{*}^\infty\right)},
\end{split}
\end{equation}
where $\la_*^\infty$ and $\omega_*^\infty$ are defined as in Eq. \eqref{3.21}.
\end{lemma}
\begin{proof}
Noticing that $\la_*^\infty\left(\omega_{*}^\infty\right)^2=c^2$ from Eq. \eqref{rela}, by virtue of Eq. \eqref{inf}, we see that
\begin{equation*}
\begin{split}
\gamma_1^\infty=&\f{2(1-\beta\la_{*}^\infty)}{\left(1+\la_{*}^\infty\right)^2}-\f{2(1-\beta\la_{*}^\infty)}{\la_{*}^\infty\left(1+\la_{*}^\infty\right)}\left(\ds\f{\la_{*}^\infty}{1+\la_{*}^\infty}-i\ds\f{\omega_{*}^\infty}{1-\beta\la_{*}^\infty}\right)\\
=&\ds\f{2i\omega_*^\infty}{\la_{*}^\infty\left(1+\la_{*}^\infty\right)},\\
\gamma_2^\infty=&-\ds\f{2c}{\la_{*}^\infty}\left[\ds\f{1}{1+\la_{*}^\infty}+i\ds\f{c}{\omega_{*}^\infty\left(1+\la_{*}^\infty\right)}\right]^2
=-\ds\f{2c}{\la_{*}^\infty\left(1+\la_{*}^\infty\right)^2}\left(1+i\ds\f{c}{\omega_{*}^\infty}\right)^2\\
=&-\ds\f{2c}{\la_{*}^\infty\left(1+\la_{*}^\infty\right)^2}\left(1-\la_{*}^\infty+2i\ds\f{c}{\omega_*^\infty}\right),\\
\end{split}
\end{equation*}
\begin{equation*}
\begin{split}
\gamma_3^\infty=&\f{2(1-\beta\la_{*}^\infty)}{\left(1+\la_{*}^\infty\right)^2}-\f{(1-\beta\la_{*}^\infty)}{\la_{*}^\infty\left(1+\la_{*}^\infty\right)}\ds\f{2\la_{*}^\infty}{1+\la_{*}^\infty}=0,\\
\gamma_4^\infty=&-\f{2c}{\la_{*}^\infty}\left|\ds\f{1}{1+\la_{*}^\infty}+i\ds\f{c}{\omega_{*}^\infty\left(1+\la_{*}^\infty\right)}\right|^2=-\ds\f{2c}{\la_{*}^\infty\left(1+\la_{*}^\infty\right)}.
\end{split}
\end{equation*}
This completes the proof.
\end{proof}
Then, we estimate the limits of $a_i$ and $b_i$ as $\ell\to\infty$ for $i=1,2,3,4$.
\begin{lemma}\label{333}
Denote
\begin{equation}
a_i^\infty=\lim_{\ell\to\infty}a_i\;\;\text{and}\;\;b_i^\infty=\lim_{\ell\to\infty}b_i\;\;\text{for}\;\;i=1,2,3,4,
\end{equation}
where $a_i,\;b_i\;(i=1,2,3,4)$ are defined as in Eq. \eqref{abv}.
Then
\begin{equation}\label{aibi}
\begin{split}
a_1^\infty=&\ds\f{1}{3\la_*^\infty}-\ds\f{ci}{\la_*^\infty\omega_*^\infty\left(1+\la_{*}^\infty\right)},\;\;
a_2^\infty=\ds\f{\la_*^\infty-5}{3\left(1+\la_{*}^\infty\right)^2}-\ds\f{ic(5\la_*^\infty-1)}{3\la_{*}^\infty\omega_*^\infty\left(1+\la_{*}^\infty\right)^2},\\
a_3^\infty=&\ds\f{-3\left(\omega_*^\infty\right)^2a_1^\infty}{-3\left(\omega_*^\infty\right)^2+\beta c\la_*^\infty+2i\beta\la_*^\infty\omega_*^\infty},\;\;
a_4^\infty=\ds\f{-3\left(\omega_*^\infty\right)^2a_2^\infty+\f{\beta\la_*^\infty r_2^\infty}{2}}{-3\left(\omega_*^\infty\right)^2+\beta c\la_*^\infty+2i\beta\la_*^\infty\omega_*^\infty},\\
b_1^\infty=&\f{1}{\la_*^\infty},\;\;b_2^\infty=\f{1}{1+\la_*^\infty},\;\;
b_3^\infty=\f{c^2}{\left(\la_*^\infty\right)^2\left[\left(\omega_*^\infty\right)^2+\beta c\la_*^\infty\right]},\\
b_4^\infty=&\f{c\left(c-\beta\la_*^\infty\right)}{\la_*^\infty\left(1+\la_{*}^\infty\right)\left[\left(\omega_*^\infty\right)^2+\beta c\la_*^\infty\right]}.
\end{split}
\end{equation}
\end{lemma}
\begin{proof}
Since $\lim_{\ell\to\infty}D_2(\la_*)=\left(\omega_*^\infty\right)^2$ and $\lim_{\ell\to\infty}T_2(\la_*)=0$,
we have
\begin{equation}\label{a1}
a_1^\infty=-\ds\f{1}{6\left(\omega_*^\infty\right)^2}\left[\gamma_1^\infty(2i\omega_*^\infty+c)-\gamma_2^\infty(1-\beta\la_*^\infty)\right].
\end{equation}
It follows from Eqs. \eqref{inf} that
\begin{equation}\label{cb}
1-\beta_*^\infty=\f{c(1+\la_*^\infty)}{\la_*^\infty} \;\;\text{and}\;\; \la(\omega_*^\infty)^2=c^2.
\end{equation}
Then substituting $\gamma_1^\infty$ and $\gamma_2^\infty$ into Eq. \eqref{a1},
we see that
\begin{equation*}
\begin{split}
a_1^\infty=&-\ds\f{1}{6\left(\omega_*^\infty\right)^2}\left[\ds\f{2i\omega_*^\infty(2i\omega_*^\infty+c)}{\la_{*}^\infty\left(1+\la_{*}^\infty\right)}
+\ds\f{2c^2}{\left(\la_{*}^\infty\right)^2\left(1+\la_{*}^\infty\right)}\left(1-\la_{*}^\infty+2i\ds\f{c}{\omega_*^\infty}\right)\right]\\
=&\ds\f{1}{3\la_*^\infty}-\ds\f{ci}{\la_*^\infty\omega_*^\infty\left(1+\la_{*}^\infty\right)}.
\end{split}
\end{equation*}
Similarly, we have
\begin{equation*}
\begin{split}
a_2^\infty=&-\ds\f{1}{6\left(\omega_*^\infty\right)^2}\left[-\ds\f{2c(2i\omega_*^\infty-c)}{\la_{*}^\infty\left(1+\la_{*}^\infty\right)^2}\left(1-\la_{*}^\infty+2i\ds\f{c}{\omega_*^\infty}\right)+\ds\f{2ic\omega_*^\infty}{\la_{*}^\infty\left(1+\la_{*}^\infty\right)}\right]\\
=&\ds\f{\la_*^\infty-5}{3\left(1+\la_{*}^\infty\right)^2}-\ds\f{ic(5\la_*^\infty-1)}{3\la_{*}^\infty\omega_*^\infty\left(1+\la_{*}^\infty\right)^2}.
\end{split}
\end{equation*}
Note that $\lim_{\ell\to\infty}T_0(\la_*)=T_0(\la_*^\infty)=-\beta\la_*^\infty$ and $\lim_{\ell\to\infty}D_0(\la_*)=D_0(\la_*^\infty)=\left(\omega_*^\infty\right)^2+\beta c\la_*^\infty$. It follows that
\begin{equation}
\begin{split}
a_3^\infty=&\ds\f{-3\left(\omega_*^\infty\right)^2a_1^\infty}{-3\left(\omega_*^\infty\right)^2+\beta c\la_*^\infty+2i\beta\la_*^\infty\omega_*^\infty},\\
a_4^\infty=&\ds\f{-3\left(\omega_*^\infty\right)^2a_2^\infty+\f{\beta\la_*^\infty r_2^\infty}{2}}{-3\left(\omega_*^\infty\right)^2+\beta c\la_*^\infty+2i\beta\la_*^\infty\omega_*^\infty}.\\
\end{split}
\end{equation}
Then we consider the limits of $b_i$ as $\ell\to\infty$ for $i=1,2,3,4$.
Noticing that $\gamma_3^\infty=0$, by virtue of
Eq. \eqref{cb}, we have
\begin{equation}
\begin{split}
b_1^\infty=&\ds\f{1}{2\left(\omega_*^\infty\right)^2}\ds\f{2c}{\la_{*}^\infty\left(1+\la_{*}^\infty\right)}
\f{c\left(1+\la_*^\infty\right)}{\la_*^\infty}=\f{1}{\la_*^\infty},\\
b_2^\infty=&\ds\f{c}{2\left(\omega_*^\infty\right)^2}\ds\f{2c}{\la_{*}^\infty\left(1+\la_{*}^\infty\right)}=\f{1}{1+\la_*^\infty},\\
b_3^\infty=&\ds\f{1}{2\left[\left(\omega_*^\infty\right)^2+\beta c\la_*^\infty\right]}\ds\f{2c}{\la_{*}^\infty\left(1+\la_{*}^\infty\right)}
\f{c\left(1+\la_*^\infty\right)}{\la_*^\infty}=\f{c^2}{\left(\la_*^\infty\right)^2\left[\left(\omega_*^\infty\right)^2+\beta c\la_*^\infty\right]},\\
b_4^\infty=&\ds\f{c-\beta\la_*^\infty}{2\left[\left(\omega_*^\infty\right)^2+\beta c\la_*^\infty\right]}\ds\f{2c}{\la_{*}^\infty\left(1+\la_{*}^\infty\right)}=\f{c\left(c-\beta\la_*^\infty\right)}{\la_*^\infty\left(1+\la_{*}^\infty\right)\left[\left(\omega_*^\infty\right)^2+\beta c\la_*^\infty\right]}.\\
\end{split}
\end{equation}
\end{proof}

Now, for simplicity of notations, let $A_1$, $A_2$ and $A_3$ satisfy
\begin{equation}\label{Ai1}
\begin{split}
\ds\f{\ell\pi D}{2}A_1=&
\ds\f{2(1-\beta\la_{*})}{\la_{*}(1+\la_{*})}\int_0^{\ell\pi}\cos^2\f{x}{\ell}\left(\ds\f{2\la_*}{1+\la_*}w_{11}^{(1)}-w_{11}^{(2)}-q_2w_{11}^{(1)}\right)dx\\
-&\ds\f{2\beta}{\ell\pi}\int_0^{\ell\pi}\cos^2\f{x}{\ell}dx\int_0^{\ell\pi}w_{11}^{(1)}dx
-\f{4c}{\la_{*}}\overline {q_2^*}(1-q_2)\int_0^{\ell\pi}\cos^2\f{x}{\ell}(w_{11}^{(1)}-w_{11}^{(2)})dx,
\end{split}
\end{equation}
\begin{equation}\label{Ai2}
\begin{split}
\ds\f{\ell\pi D}{2}A_2
=&\left[-\f{6(1-\beta\la_{*})}{(1+\la_{*})^3}+\ds\f{2(1-\beta\la_{*})}{\la_{*}(1+\la_{*})^2}
\left(\overline q_2+2q_2\right)\right]\int_0^{\ell\pi}\cos^4\f{x}{\ell}dx\\
+&\left[-\f{4c}{\la_{*}^2}\overline {q_2^*}(\overline q_2+2q_2)+\f{6c}{\la_{*}^2}\overline {q_2^*}+\f{2c}{\la_{*}^2}\overline {q_2^*}(q_2^2+2q_2\overline q_2)\right]\int_0^{\ell\pi}\cos^4\f{x}{\ell}dx,
\end{split}
\end{equation}
and
\begin{equation}\label{Ai3}
\begin{split}
\ds\f{\ell\pi D}{2}A_3
=&\ds\f{2(1-\beta\la_{*})}{\la_{*}(1+\la_{*})}\int_0^{\ell\pi}\cos^2\f{x}{\ell}\left(
\ds\f{\la_*w_{20}^{(1)}}{1+\la_*}-\ds\f{w_{20}^{(2)}}{2}-\overline q_2\ds\f{w_{20}^{(1)}}{2}\right)dx\\
-&\f{\beta}{\ell\pi}\int_0^{\ell\pi}\cos^2\f{x}{\ell}dx\int_0^{\ell\pi}w_{20}^{(1)}dx
-\f{2c}{\la_{*}}\overline {q_2^*}(1-\overline q_2)\int_0^{\ell\pi}\cos^2\f{x}{\ell}(w_{20}^{(1)}-w_{20}^{(2)})dx.\\
\end{split}
\end{equation}
Then $g_{21}=A_1+A_2+A_3$. Based on Lemmas \ref{ll3}-\ref{333}, we first give the estimate of $A_1$ as $\ell\to\infty$.
\begin{lemma}\label{esta1}
Denote $A_1^\infty=\lim_{\ell\to\infty} A_1$, where $A_1$ is defined as in Eq. \eqref{Ai1}. Then
\begin{equation*}
{\rm Re}(A_1^\infty)=\ds\f{7c}{2\left(\la_*^\infty\right)^2(1+\la_*^\infty)}
+\ds\f{c\left[\left(\omega_*^\infty\right)^2-\beta c\la_*^\infty\right]}{\left(\la_*^\infty\right)^2(1+\la_*^\infty)\left[\left(\omega_*^\infty\right)^2+\beta c\la_*^\infty\right]}.
\end{equation*}
\end{lemma}
\begin{proof}
Note that
\begin{equation}\label{ine}
\int_0^{\ell\pi}\cos^2\ds\f{x}{\ell}\cos\ds\f{2x}{\ell}dx=\ds\f{\ell\pi}{4},
\;\;\int_0^{\ell\pi}\cos^2\ds\f{x}{\ell}dx=\ds\f{\ell\pi}{2},\;\;\int_0^{\ell\pi}\cos^4\ds\f{x}{\ell}dx=\ds\f{3\ell\pi}{8}.
\end{equation}
It follows from Eq. \eqref{Ai1} that
\begin{equation}\label{above}
\begin{split}
A_1=&\ds\f{(1-\beta\la_{*})}{D\la_{*}(1+\la_{*})}\left(\ds\f{2\la_*}{1+\la_*}b_1-b_2-q_2b_1\right)
+\ds\f{2(1-\beta\la_{*})}{D\la_{*}(1+\la_{*})}\left(\ds\f{2\la_*}{1+\la_*}b_3-b_4-q_2b_3\right)\\
-&\ds\f{2\beta b_3}{D}-\f{2c}{D\la_{*}}\overline {q_2^*}(1-q_2)(b_1-b_2)-\f{4c}{D\la_{*}}\overline {q_2^*}(1-q_2)(b_3-b_4).
\end{split}
\end{equation}
By virtue of Eq. \eqref{cb}, and taking the limits at both sides of Eq. \eqref{above} as $\ell\to\infty$, we see that
\begin{equation*}
\begin{split}
&{\rm Re}(A_1^\infty)\\=&\ds\f{7c}{2\left(\la_*^\infty\right)^2(1+\la_*^\infty)}
+\ds\f{c\left[\left(\omega_*^\infty\right)^2+\beta c\right]}{\left(\la_*^\infty\right)^2(1+\la_*^\infty)\left[\left(\omega_*^\infty\right)^2+\beta c\la_*^\infty\right]}
-\f{\beta c^2}{\left(\la_*^\infty\right)^2\left[\left(\omega_*^\infty\right)^2+\beta c\la_*^\infty\right]}\\
=&\ds\f{7c}{2\left(\la_*^\infty\right)^2(1+\la_*^\infty)}
+\ds\f{c\left[\left(\omega_*^\infty\right)^2-\beta c\la_*^\infty\right]}{\left(\la_*^\infty\right)^2(1+\la_*^\infty)\left[\left(\omega_*^\infty\right)^2+\beta c\la_*^\infty\right]}.
\end{split}
\end{equation*}
\end{proof}
Similarly, we obtain the estimate of $A_2$ as $\ell\to\infty$.
\begin{lemma}\label{esta2}
Denote $A_2^\infty=\lim_{\ell\to\infty} A_2$, where $A_2$ is defined as in Eq. \eqref{Ai2}. Then
\begin{equation*}
{\rm Re}(A_2^\infty)=-\ds\f{3c}{4\la_*^\infty(1+\la_*^\infty)^2}-\ds\f{3c}{2\left(\la_*^\infty\right)^2(1+\la_*^\infty)}.
\end{equation*}
\end{lemma}
\begin{proof}
It follows from \eqref{Ai2} and \eqref{ine} that
\begin{equation}\label{above2}
\begin{split}
A_2
=&\ds\f{3}{4D}\left[-\f{6(1-\beta\la_{*})}{(1+\la_{*})^3}+\ds\f{2(1-\beta\la_{*})}{\la_{*}(1+\la_{*})^2}
\left(\overline q_2+2q_2\right)\right]\\
+&\ds\f{3}{4D}\left[-\f{4c}{\la_{*}^2}\overline {q_2^*}(\overline q_2+2q_2)+\f{6c}{\la_{*}^2}\overline {q_2^*}+\f{2c}{\la_{*}^2}\overline {q_2^*}(q_2^2+2q_2\overline q_2)\right].
\end{split}
\end{equation}
By virtue of Eq. \eqref{cb}, and taking the limits at both sides of Eq. \eqref{above2} as $\ell\to\infty$, we see that
\begin{equation*}
\begin{split}
{\rm Re}(A_2^\infty)
=&-\ds\f{3c}{4\la_*^\infty(1+\la_*^\infty)^2}-\ds\f{3c}{2(\la_*^\infty)^2}+\ds\f{3c}{2\la_*^\infty(1+\la_*^\infty)}\\
=&-\ds\f{3c}{4\la_*^\infty(1+\la_*^\infty)^2}-\ds\f{3c}{2\left(\la_*^\infty\right)^2(1+\la_*^\infty)}.
\end{split}
\end{equation*}
\end{proof}

Then we consider the estimate of $A_3$ as $\ell\to\infty$, which is more complex than that of $A_1$ and $A_2$.
\begin{lemma}\label{esta3}
Denote $A_3^\infty=\lim_{\ell\to\infty} A_3$, where $A_3$ is defined as in Eq. \eqref{Ai3}. Then
\begin{equation}\label{a3}
\begin{split}
{\rm Re}(A_3^\infty)
=&-\ds\f{9c}{4\left(\la_*^\infty\right)^2(1+\la_*^\infty)}\\
+&\ds\f{\beta c^2}{2(\la_*^\infty)^2(\la_*^\infty+1)}\left(\ds\f{\beta c\la_*^\infty(\la_*^\infty+4)-(\omega_*^\infty)^2(\la_*^\infty-2)}{\left(-3(\omega_*^\infty)^2+\beta c\la_*^\infty\right)^2+4\beta^2 (\la_*^\infty)^2(\omega_*^\infty)^2}\right).\\
\end{split}
\end{equation}
\end{lemma}
\begin{proof}
It follows from Eqs. \eqref{Ai3} and \eqref{ine} that
\begin{equation}\label{above3}
\begin{split}
A_3=&\ds\f{c}{D\left(\la_{*}\right)^2}\left[\left(\ds\f{2\la_*}{1+\la_*}-\overline q_2\right)\ds\f{a_1}{2}-\ds\f{a_2}{2}\right]
+\ds\f{2c}{D\left(\la_{*}\right)^2}\left[\left(\ds\f{2\la_*}{1+\la_*}-\overline q_2\right)\ds\f{a_3}{2}-\ds\f{a_4}{2}\right]\\
-&\ds\f{\beta a_3}{D}-\f{c}{D\la_{*}}\overline {q_2^*}(1-\overline q_2)(a_1-a_2)-\f{2c}{D\la_{*}}\overline {q_2^*}(1-\overline q_2)(a_3-a_4).
\end{split}
\end{equation}
Taking the limits at both sides of Eq. \eqref{above3} as $\ell\to\infty$, we see that
$A_3^\infty=B_1^\infty+B_2^\infty+B_3^\infty+B_4^\infty+B_5^\infty$,
where
\begin{equation}\label{Bi}
\begin{split}
B_1^\infty=&\ds\f{c}{D^\infty\left(\la_{*}^\infty\right)^2}\left[\left(\ds\f{2\la_*^\infty}{1+\la_*^\infty}-\overline {q_2^\infty}\right)\ds\f{a_1^\infty}{2}-\ds\f{a_2^\infty}{2}\right],\\
B_2^\infty=&\ds\f{2c}{D^\infty\left(\la_{*}^\infty\right)^2}\left[\left(\ds\f{2\la_*^\infty}{1+\la_*^\infty}-\overline {q_2^\infty}\right)\ds\f{a_3^\infty}{2}-\ds\f{a_4^\infty}{2}\right],\\
B_3^\infty=&-\ds\f{\beta a_3^\infty}{D^\infty},\;\;B_4^\infty=-\f{c}{D^\infty\la_{*}^\infty}\overline {q_2^{*\infty}}(1-\overline{q_2^\infty})(a_1^\infty-a_2^\infty),\\
B_5^\infty=&-\f{2c}{D^\infty\la_{*}^\infty}\overline {q_2^{*\infty}}(1-\overline{q_2^\infty})(a_3^\infty-a_4^\infty).
\end{split}
\end{equation}
A direct calculation yields
\begin{equation}\label{B14}
\begin{split}
B_1^\infty
=&\ds\f{c}{12\left(\la_{*}^\infty\right)^2\left(1+\la_*^\infty\right)}-\ds\f{i\omega_*^\infty}{6\left(\la_{*}^\infty\right)^2\left(1+\la_*^\infty\right)},\\
B_4^\infty=&\ds\f{-5c}{6\left(\la_{*}^\infty\right)^2\left(1+\la_*^\infty\right)}-i\ds\f{\omega_*^\infty\left(2\la_*^\infty+1\right)}{6\left(\la_{*}^\infty\right)^2\left(1+\la_*^\infty\right)}.
\end{split}
\end{equation}
Denote $\sigma=\ds\f{1}{-3\left(\omega_*^\infty\right)^2+\beta c\la_*^\infty+2i\beta\la_*^\infty\omega_*^\infty}$.
Then
\begin{equation}\label{B235}
\begin{split}
&B_2^\infty+B_3^\infty+B_5^\infty\\
=&\ds\f{-6\left(\omega_*^\infty\right)^2}{\sigma}(B_1^\infty+B_4^\infty)-\ds\f{\beta c \gamma_2^\infty}{2\sigma D^\infty\la_*^\infty}+\ds\f{3\beta \left(\omega_*^\infty\right)^2 a_1^\infty}{\sigma D^\infty}+\ds\f{\beta c\gamma_2^\infty\overline {q_2^{*\infty}}}{\sigma D^\infty}(1-\overline{q_2^\infty})\\
=&\ds\f{\left(\omega_*^\infty\right)^2}{\sigma}\left[\ds\f{9c}{2\left(\la_*^\infty\right)^2(\la_*^\infty+1)}+i\ds\f{2\omega_*^\infty}{\left(\la_*^\infty\right)^2}\right]
+\ds\f{\beta c^2}{2\sigma \left(\la_*^\infty\right)^2(\la_*^\infty+1)}\left(1+i\ds\f{c}{\omega_*^\infty}\right)\\
+&\ds\f{\beta c^2}{2\sigma (\la_*^\infty)^2(\la_*^\infty+1)}\left[1-2\la_*^\infty-i\ds(\la_*^\infty+4)\f{c}{\omega_*^\infty}\right]
+\ds\f{\beta c^2}{\sigma (\la_*^\infty)^2(\la_*^\infty+1)}\left(\la_*^\infty-i\ds\f{c}{\omega_*^\infty}\right)\\
=&\ds\f{\left(\omega_*^\infty\right)^2}{\sigma}\left[\ds\f{9c}{2\left(\la_*^\infty\right)^2(\la_*^\infty+1)}
+i\ds\f{2\omega_*^\infty}{\left(\la_*^\infty\right)^2}\right]+\ds\f{\beta c^2}{2\sigma \left(\la_*^\infty\right)^2(\la_*^\infty+1)}\left[2-i\ds(\la_*^\infty+5)\f{c}{\omega_*^\infty}\right]\\
=&-\ds\f{3c}{2\left(\la_*^\infty\right)^2(\la_*^\infty+1)}
-i\ds\f{2\omega_*^\infty}{3\left(\la_*^\infty\right)^2}+\ds\f{\beta c^2}{2\sigma \left(\la_*^\infty\right)^2(\la_*^\infty+1)}\left[2-i\ds(\la_*^\infty+5)\f{c}{\omega_*^\infty}\right]\\
-&\ds\f{\beta c\la_*^\infty+2i\beta\la_*^\infty\omega_*^\infty}{\sigma}\left[-\ds\f{3c}{2\left(\la_*^\infty\right)^2(\la_*^\infty+1)}
-i\ds\f{2\omega_*^\infty}{3\left(\la_*^\infty\right)^2}\right]\\
=&-\ds\f{3c}{2\left(\la_*^\infty\right)^2(\la_*^\infty+1)}
-i\ds\f{2\omega_*^\infty}{3\left(\la_*^\infty\right)^2}\\
+&\ds\f{\beta c^2}{2\sigma \left(\la_*^\infty\right)^2(\la_*^\infty+1)}\left[\ds\f{\la_*^\infty-2}{3}+i\left(\f{\la_*^\infty+7}{3}\right)\f{c}{\omega_*^\infty}\right].
\end{split}
\end{equation}
Therefore, Eq. \eqref{a3} is derived.
\end{proof}
Summarizing $A_1$, $A_2$ and $A_3$, we can obtain the estimate of $g_{21}$ as $\ell\to\infty$.
\begin{theorem}\label{Th37}
Denote $g_{21}^\infty(\la_*^\infty)=\lim_{\ell\to\infty} g_{21}$, where $g_{21}$ is defined as in Eq. \eqref{g}. Then
\begin{equation}\label{g21}
\begin{split}
&{\rm Re}(g_{21}^\infty(\la_*^\infty))\\=&-\ds\f{3c}{4\la_*^\infty(1+\la_*^\infty)^2}-\ds\f{c}{4\left(\la_*^\infty\right)^2(1+\la_*^\infty)}+\ds\f{c\left[\left(\omega_*^\infty\right)^2-\beta c\la_*^\infty\right]}{\left(\la_*^\infty\right)^2(1+\la_*^\infty)\left[\left(\omega_*^\infty\right)^2+\beta c\la_*^\infty\right]}\\
+&\ds\f{\beta c^2}{2(\la_*^\infty)^2(\la_*^\infty+1)}\left(\ds\f{\beta c\la_*^\infty(\la_*^\infty+4)-(\omega_*^\infty)^2(\la_*^\infty-2)}{\left(-3(\omega_*^\infty)^2+\beta c\la_*^\infty\right)^2+4\beta^2 (\la_*^\infty)^2(\omega_*^\infty)^2}\right),\\
\end{split}
\end{equation}
where $\la_*=\la_{1,+}^H\;\text{or}\;\la_{1,-}^H$.
\end{theorem}

Note that $\mu'(\la_{1,+}^H)<0$ and $\mu'(\la_{1,-}^H)>0$.
We obtain the stability on the bifurcating spatially nonhomogeneous periodic solutions for large spatial scale $\ell$.
\begin{theorem}\label{Th38}
Assume that $d_1/d_2>p_1(\la_1)$ and $c<p_2(\la_2)$. Then, for sufficiently large $\ell$, $(\la,\la)$ is locally asymptotically stable for $\la\in\left(\la_{1,+}^H,1/\beta\right)\cup\left(0,\la_{1,-}^H\right)$ and unstable for $\la\in\left(\la_{1,-}^H,\la_{1,+}^H\right)$, and system \eqref{nonlocal3} undergoes Hopf bifurcation at $(\la,\la)$ when $\la=\la_{1,+}^H$ or $\la=\la_{1,-}^H$, where $\la_{1,-}^H$ and $\la_{1,+}^H$ are defined as in Eq. \eqref{Hopp}. Moreover,
\begin{enumerate}
\item[(1)] if ${\rm Re}(g_{21}^\infty(\la_*^\infty))>0$, where $\la_*=\la_{1,+}^H$ (respectively, $\la_*=\la_{1,-}^H$), and ${\rm Re}(g_{21}^\infty(\la_*^\infty))$ is defined as in Eq. \eqref{g21}, then the bifurcating spatially nonhomogeneous periodic solutions from $\la_{1,+}^H$ (respectively, $\la_{1,-}^H$) are unstable and exist
in the right neighborhood of $\la_{1,+}^H$ (respectively, in the left neighborhood of $\la_{1,-}^H$);
\item[(2)] if ${\rm Re}(g_{21}^\infty(\la_*^\infty))<0$, where $\la_*=\la_{1,+}^H$ (respectively, $\la_*=\la_{1,-}^H$), then the bifurcating spatially nonhomogeneous periodic solutions from $\la_{1,+}^H$ (respectively, $\la_{1,-}^H$) are orbitally asymptotically stable and exist in the left neighborhood of $\la_{1,+}^H$ (respectively, in the right neighborhood of $\la_{1,-}^H$).
\end{enumerate}
\end{theorem}
In the following, we
give some numerical simulations to illustrate the obtained theoretical
results.

\begin{example}\label{exam1} To visualize the results in Theorems \ref{Hopfm} and \ref{Th38} , we choose
\begin{equation}\label{p1}
d_1=0.8,~d_2=1,~\beta=1.5,~c=0.1.
\end{equation}
Then system \eqref{nonlocal3} has a unique constant positive equilibrium $(\la,\la)$ if and only if $\la\in(0,2/3)$.
It follows from Theorem \ref{Hopfm} that, for sufficiently large $\ell$, there exist two Hopf bifurcation points $\la_{1,-}^H$ and $\la_{1,+}^H$ such that $(\la,\la)$ is locally asymptotically stable for $\la\in\left(0,\la_{1,-}^H\right)\cup\left(\la_{1,+}^H,2/3\right)$ and unstable for $\la\in\left(\la_{1,-}^H,\la_{1,+}^H\right)$.
Moreover, the bifurcating periodic solutions are spatially nonhomogeneous near these two Hopf bifurcation points $\la_{1,-}^H$ and $\la_{1,+}^H$. Note that $\la$
is equivalent to parameter $b$, where $$(1-\beta\la)(1+\la)=b\la,$$
and consequently, $b$ is strictly decreasing with respect to $\la$. Then there exist two Hopf bifurcation points $b_{1,\pm}^H$, which satisfy
\begin{equation}
b_{1,\pm}^H=\ds\f{(1-\beta\la_{1,\pm}^H)(1+\la_{1,\pm}^H)}{\la_{1,\pm}^H},
\end{equation}
such that the positive constant equilibrium of system \eqref{nonlocal3} is locally asymptotically stable for $b\in\left(0,b_{1,+}^H\right)\cup\left(b_{1,-}^H,\infty\right)$ and unstable for $b\in\left(b_{1,+}^H,b_{1,-}^H\right)$.

By virtue of Lemma \ref{3.21} and Theorem \ref{Th37}, we can easily calculate
\begin{equation*}
\lim_{\ell\to\infty}\la_{1,+}^H\approx0.4528,\;\;\lim_{\ell\to\infty}\la_{1,-}^H\approx0.1472,\;\;
\lim_{\ell\to\infty}b_{1,+}^H\approx1.0296,\;\;\lim_{\ell\to\infty}b_{1,-}^H\approx6.0704,
\end{equation*}
$\lim_{\ell\to\infty}g_{21}\approx -0.1254<0$  for Hopf bifurcation point $\la_{1,+}^H$,
and $\lim_{\ell\to\infty}g_{21}\approx 2.0724>0$ for Hopf bifurcation point $\la_{1,-}^H$.
Then, it follows from Theorem \ref{Th38} that, for sufficiently large $\ell$,
\begin{enumerate}
\item[(1)] the bifurcating spatially nonhomogeneous periodic solutions from $b_{1,+}^H$ are orbitally asymptotically stable and exist
in the right neighborhood of $b_{1,+}^H$;
\item[(2)] the bifurcating spatially nonhomogeneous periodic solutions from $b_{1,-}^H$ are unstable and exist
in the right neighborhood of $b_{1,-}^H$.
\end{enumerate}
Numerically, we show that the solution converges to the stable spatially nonhomogeneous periodic solution, which bifurcats from $b_{1,+}^H$, and the periodic solution concentrates more on the boundary of the domain when spatial scale $\ell$ increases, see Fig. \ref{hopf1}.
\begin{figure}[htbp]
\centering\includegraphics[width=0.5\textwidth]{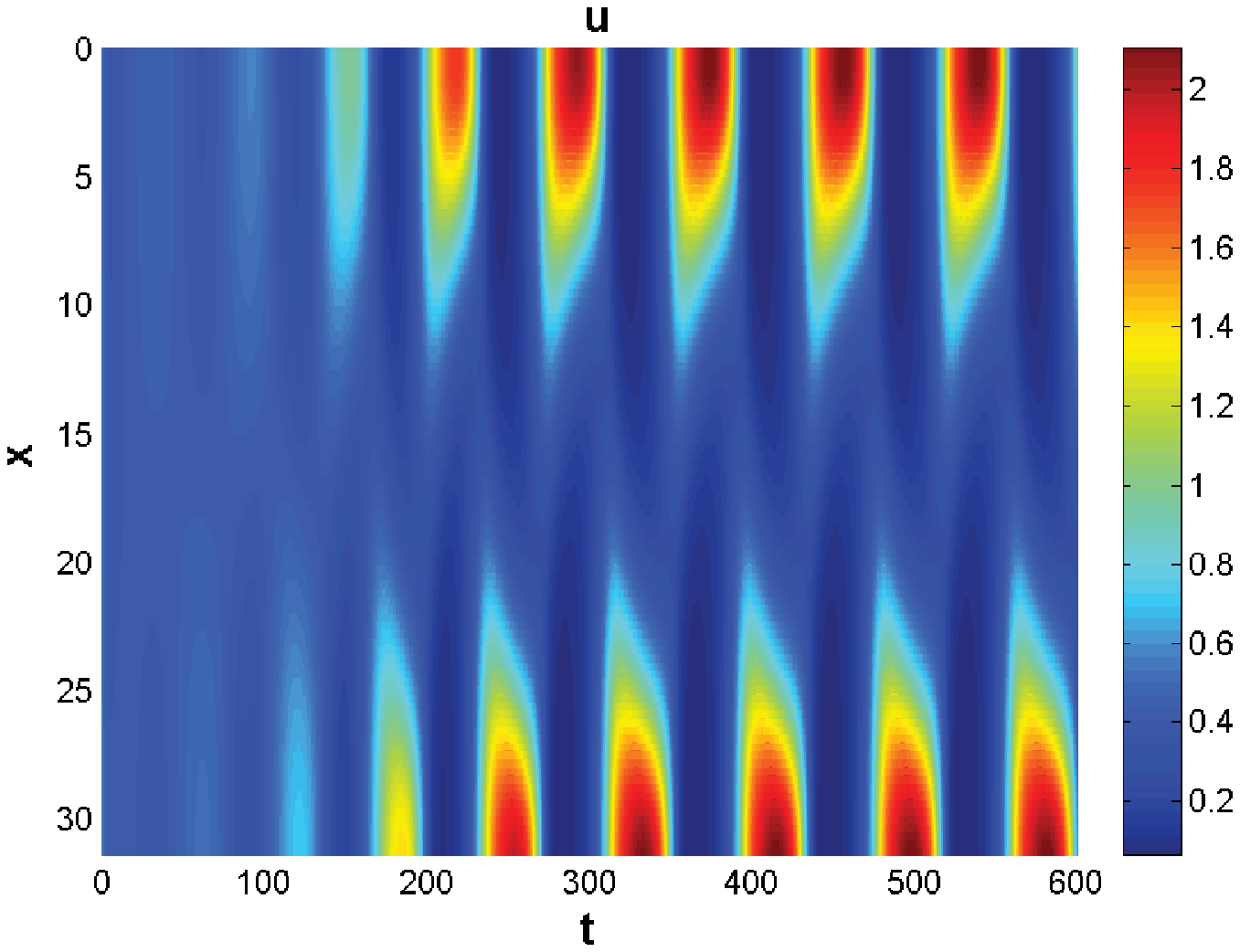}\includegraphics[width=0.5\textwidth]{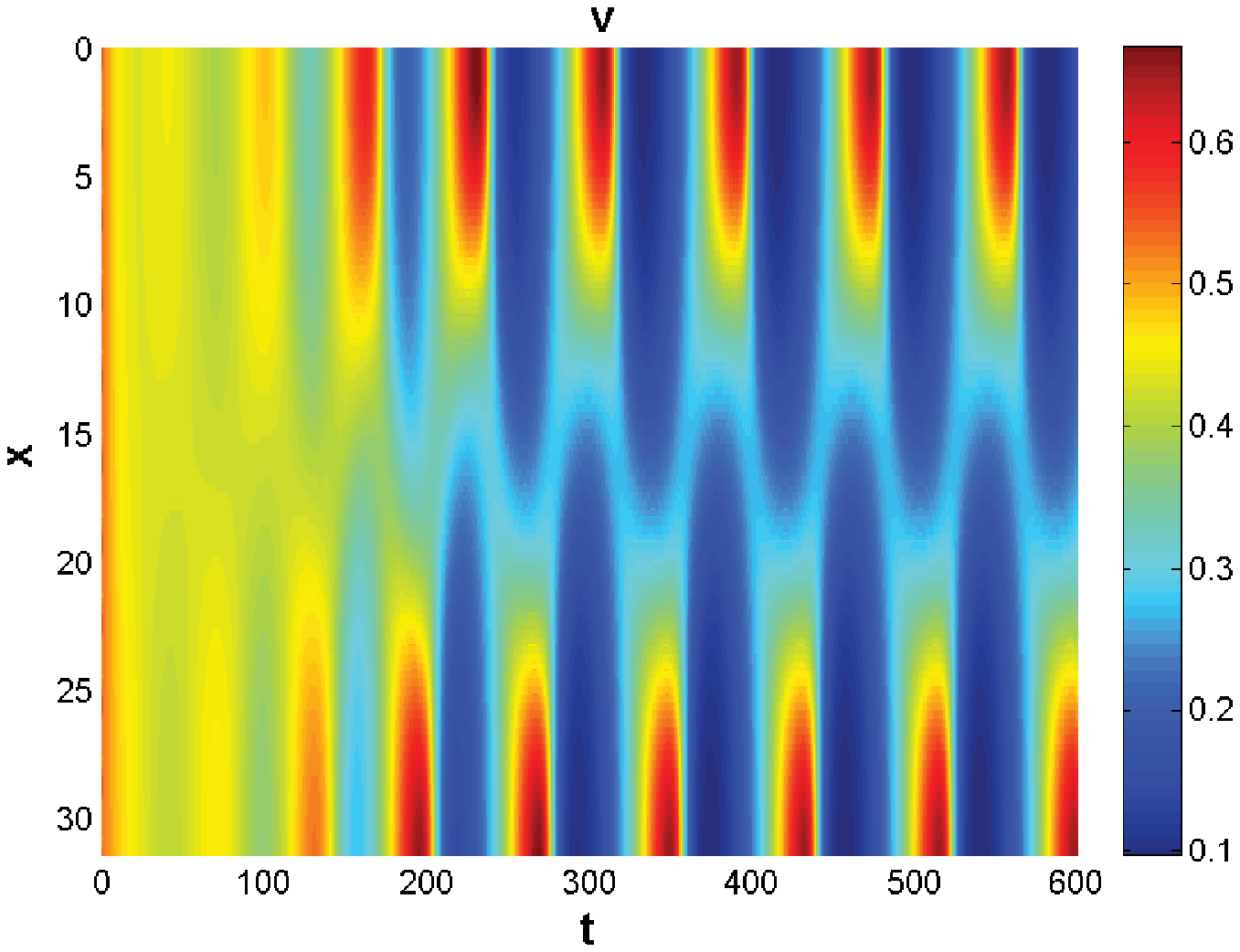}\\
\centering\includegraphics[width=0.5\textwidth]{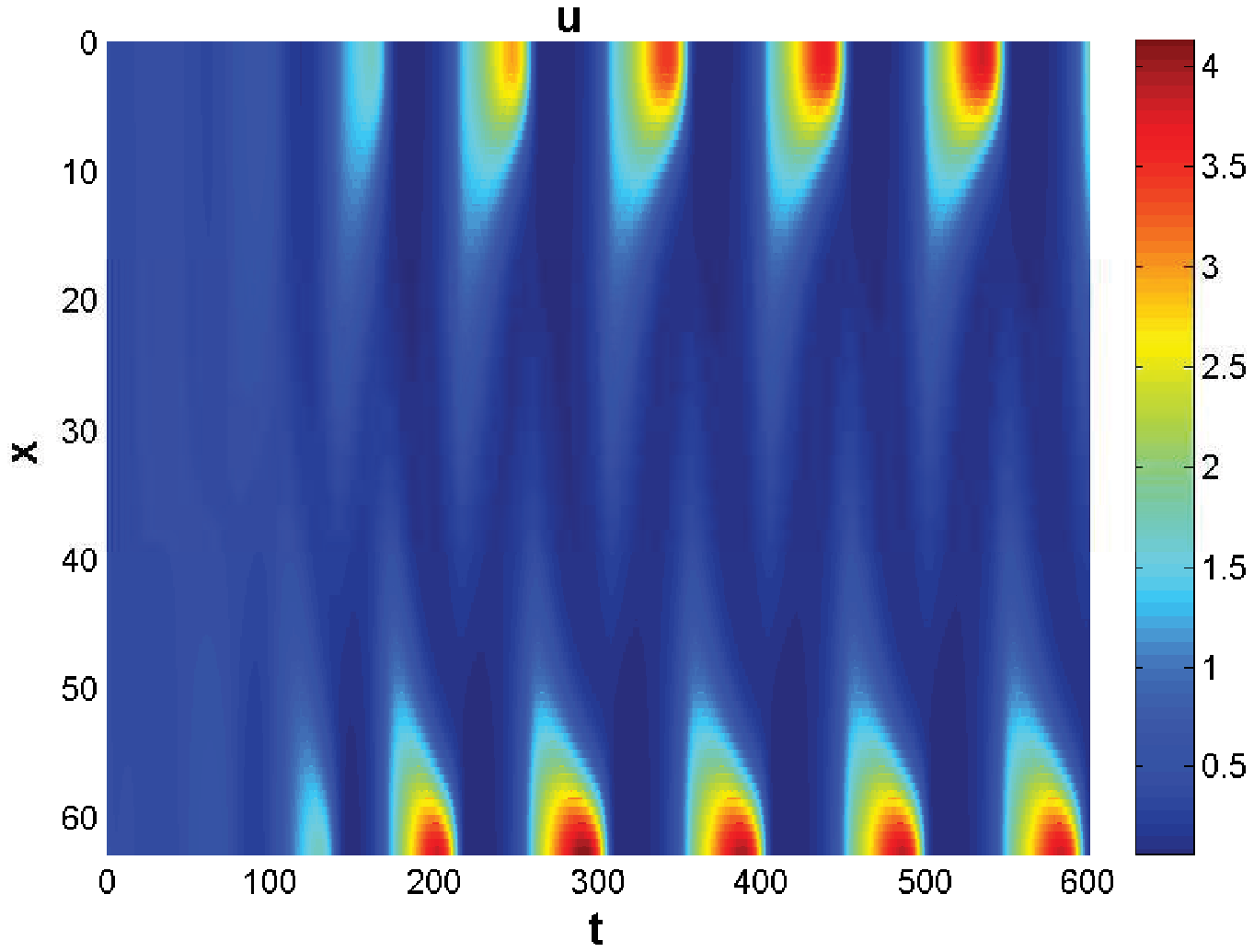}\includegraphics[width=0.5\textwidth]{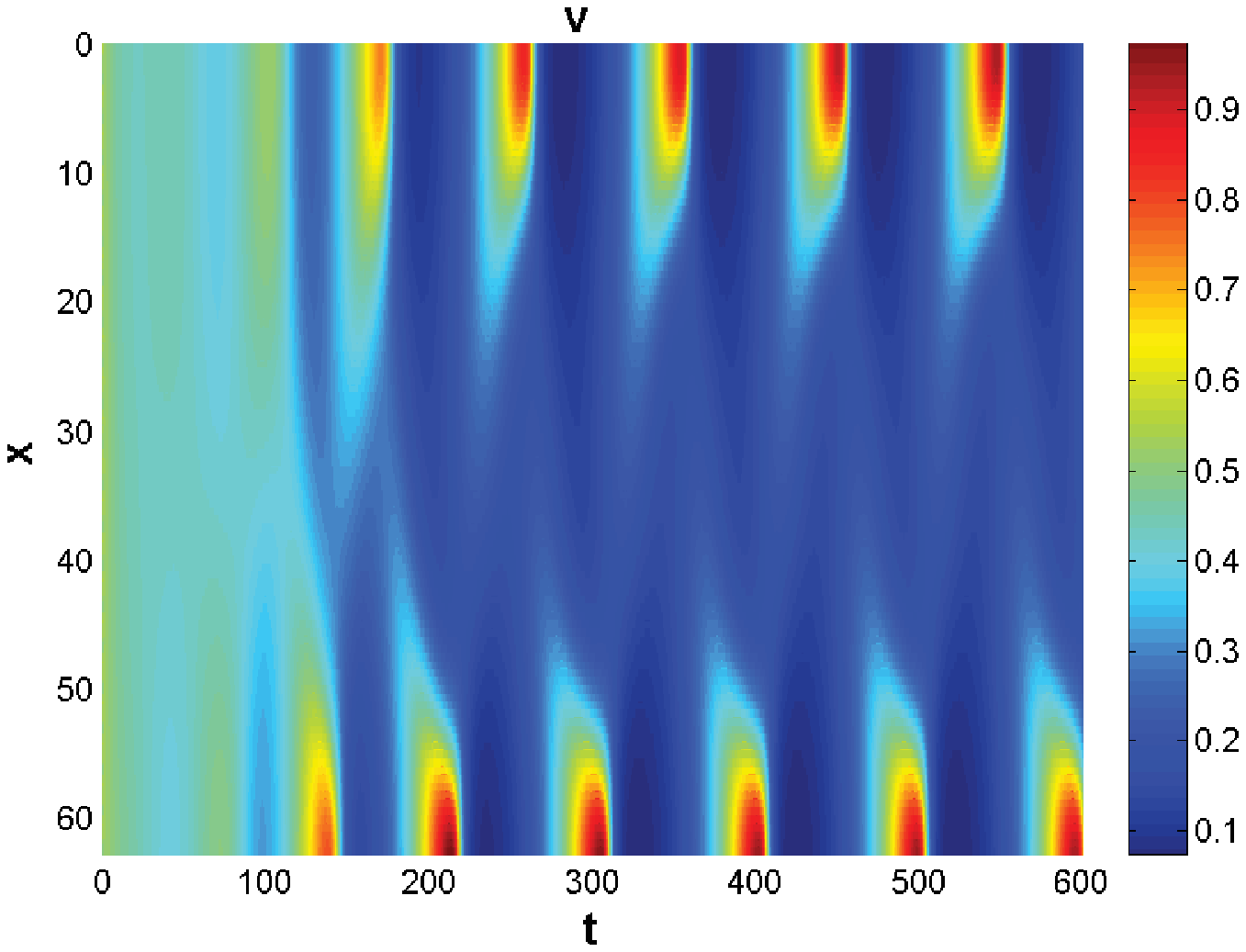}
 \caption{The solution converges to the bifurcating spatially nonhomogeneous periodic solution from $b_{1,+}^H$. Here $b=1.2$, and initial values:  $u(x,0)=0.5+\ds\f{0.05x^2}{\ell^2\pi^2},~v(x,0)=0.5+0.05\cos^2x$. (Upper): $\ell=10$;
(Lower): $\ell=20$.\label{hopf1}}
\end{figure}

\end{example}

\begin{example}\label{exam2} To visualize the results in Theorems \ref{Hopfm3} and \ref{Th38}, we choose
\begin{equation}\label{p2}
d_1=0.4,~d_2=0.6,~\beta=0.2,~c=0.2.
\end{equation}
Similarly, it follows from Theorem \ref{Hopfm3} that, for sufficiently large $\ell$, there exist two Hopf bifurcation points $\la_{1,-}^H$ and $\la_{1,+}^H$ such that $(\la,\la)$ is locally asymptotically stable for $\la\in\left(0,\la_{1,-}^H\right)\cup\left(\la_{1,+}^H,5\right)$ and unstable for $\la\in\left(\la_{1,-}^H,\la_{1,+}^H\right)$.
Moreover, the bifurcating periodic solutions are spatially nonhomogeneous near these two Hopf bifurcation points $\la_{1,-}^H$ and $\la_{1,+}^H$. Then there exist two Hopf bifurcation points $b_{1,\pm}^H$
such that the positive constant equilibrium of system \eqref{nonlocal3} is locally asymptotically stable for $b\in\left(0,b_{1,+}^H\right)\cup\left(b_{1,-}^H,\infty\right)$ and unstable for $b\in\left(b_{1,+}^H,b_{1,-}^H\right)$.

Similarly, by virtue of Lemma \ref{3.21} and Theorem \ref{Th37}, we can also easily compute
\begin{equation*}
\lim_{\ell\to\infty}\la_{1,+}^H\approx 3.7321,\;\;\lim_{\ell\to\infty}\la_{1,-}^H\approx 0.2679,\;\;
\lim_{\ell\to\infty}b_{1,+}^H\approx  0.3215,\;\;\lim_{\ell\to\infty}b_{1,-}^H\approx 4.4785,
\end{equation*}
$\lim_{\ell\to\infty}g_{21}\approx -0.0033<0$  for Hopf bifurcation point $\la_{1,+}^H$, and
$\lim_{\ell\to\infty}g_{21}\approx  1.0745>0$ for Hopf bifurcation point $\la_{1,-}^H$.
Then, it follows from Theorem \ref{Th38} that, for sufficiently large $\ell$,
\begin{enumerate}
\item[(1)] the bifurcating spatially nonhomogeneous periodic solutions from $b_{1,+}^H$ are orbitally asymptotically stable and exist
in the right neighborhood of $b_{1,+}^H$;
\item[(2)] the bifurcating spatially nonhomogeneous periodic solutions from $b_{1,-}^H$ are unstable and exist
in the right neighborhood of $b_{1,-}^H$.
\end{enumerate}
Similarly as in Example \ref{exam1}, we also numerically show that the solution converges to the stable spatially nonhomogeneous periodic solution bifurcating from $b_{1,+}^H$, and the periodic solution concentrates more on the boundary of the domain when spatial scale $\ell$ increases, see Fig. \ref{hopf2}.
\begin{figure}[htbp]
\centering\includegraphics[width=0.5\textwidth]{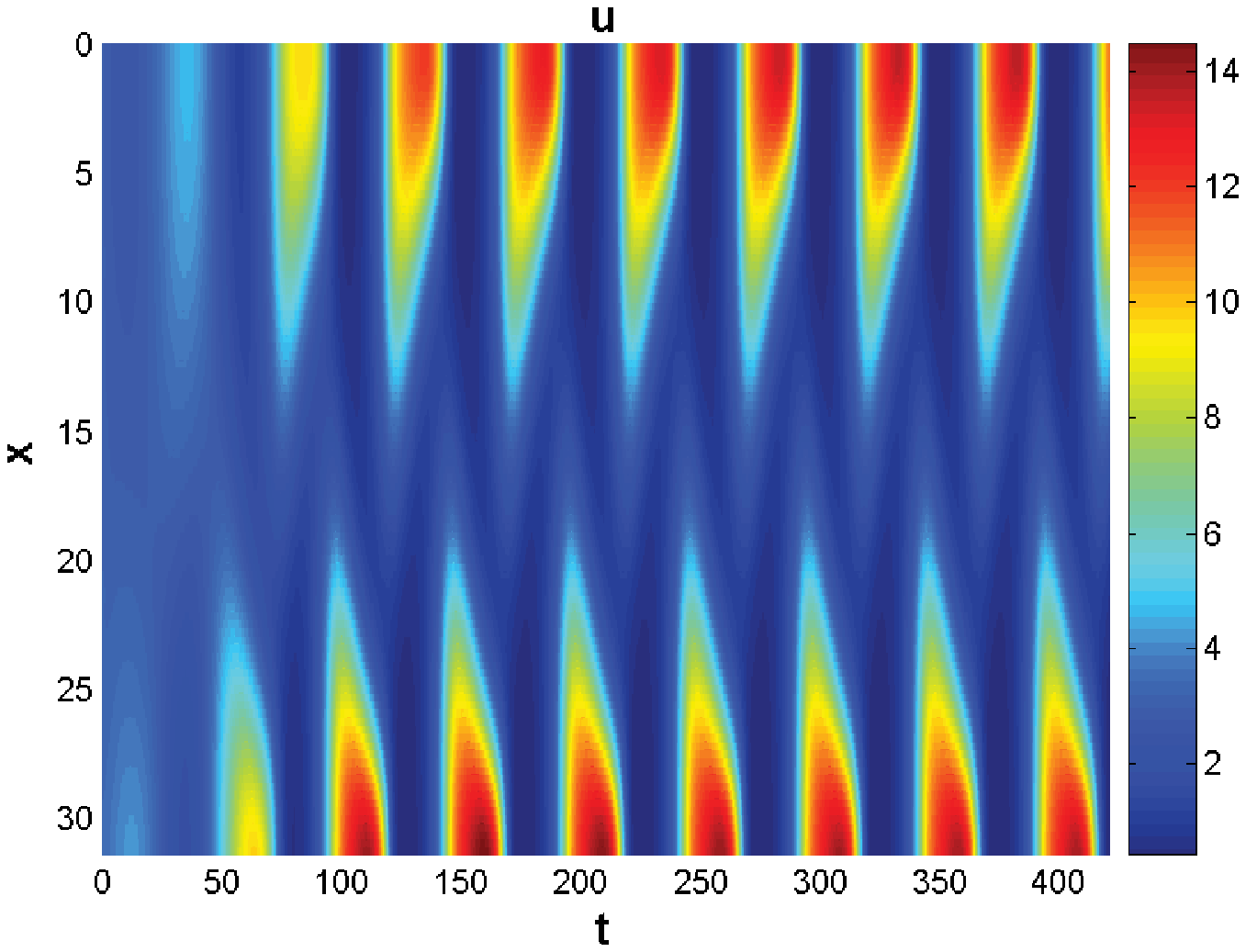}\includegraphics[width=0.5\textwidth]{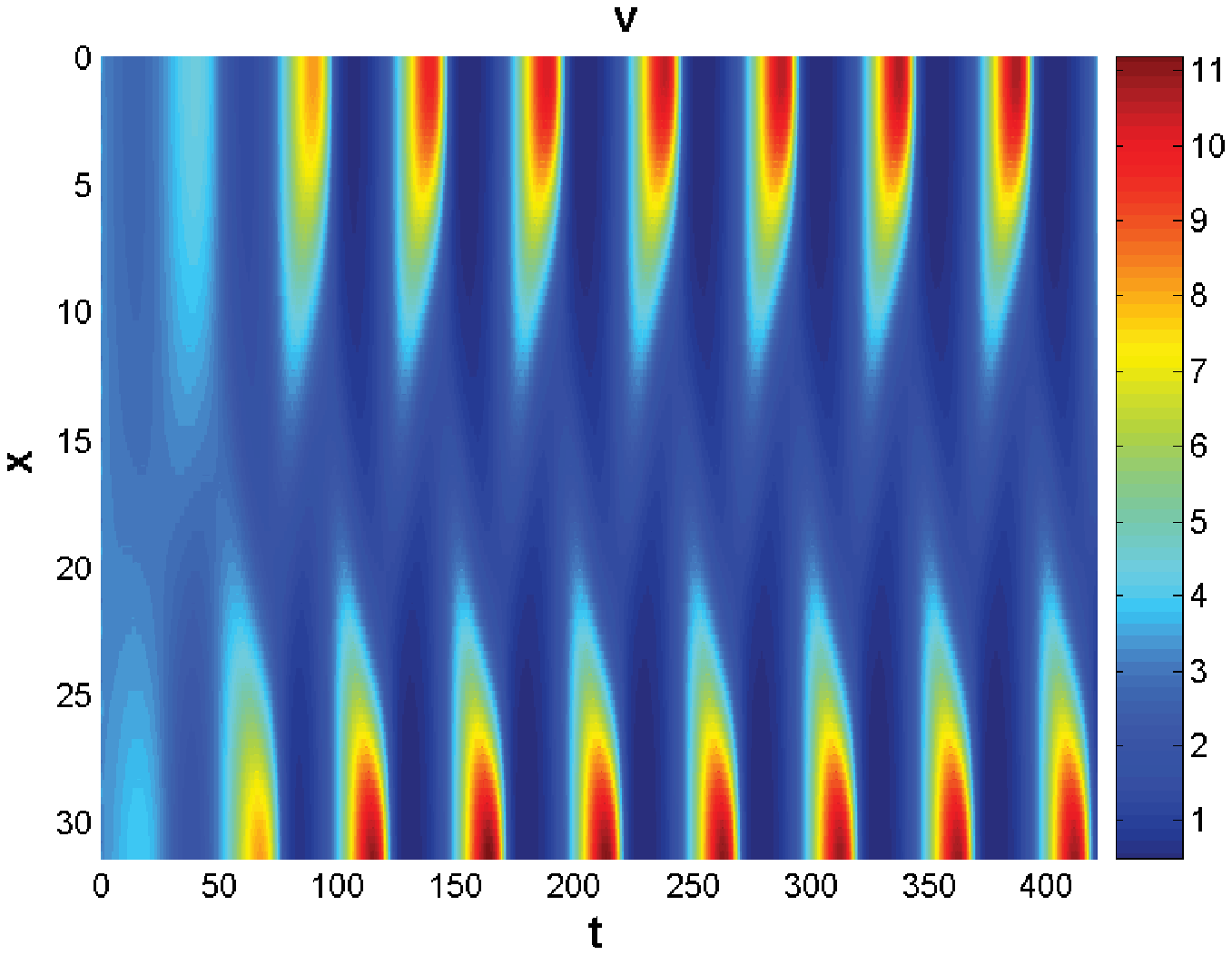}\\
\centering\includegraphics[width=0.5\textwidth]{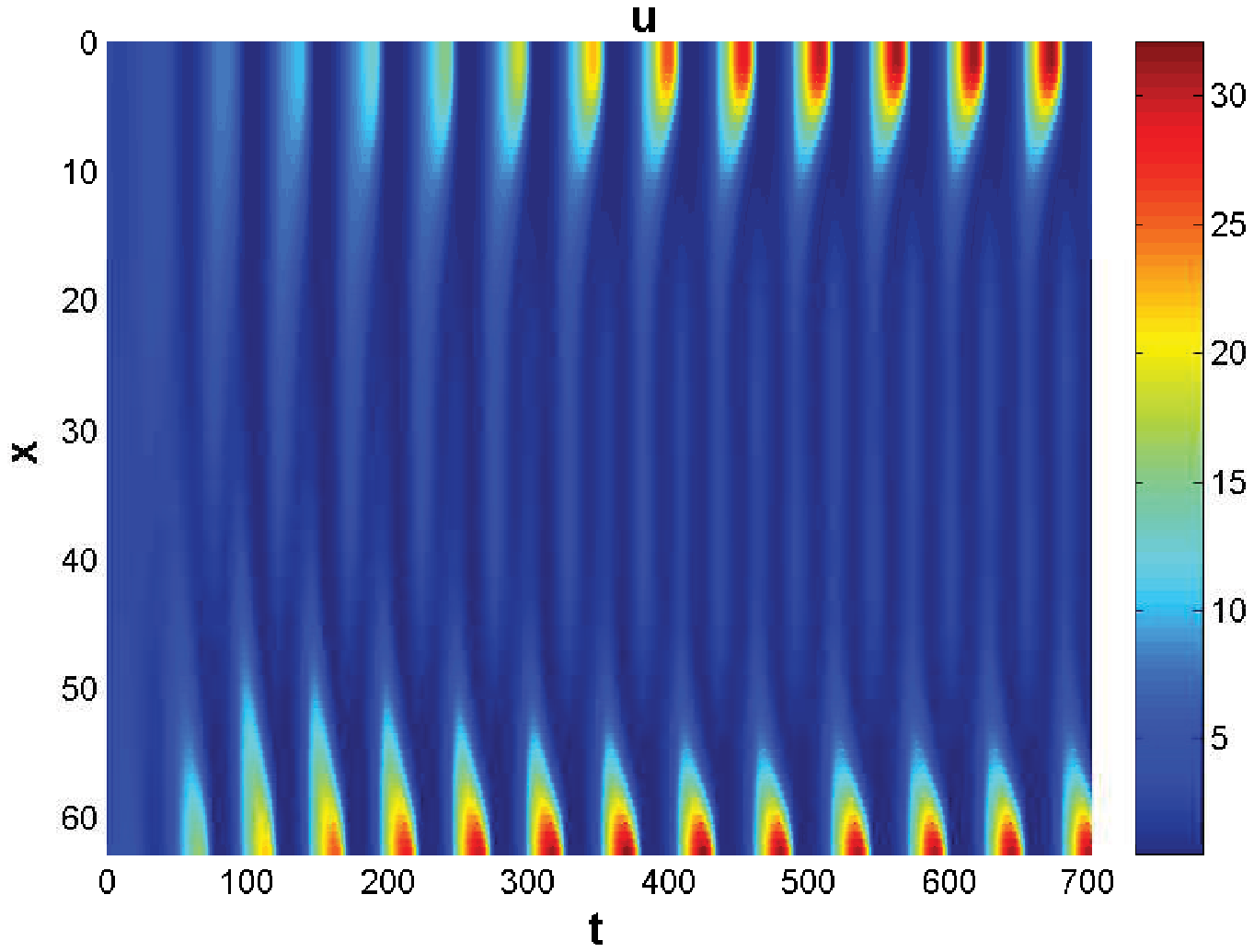}\includegraphics[width=0.5\textwidth]{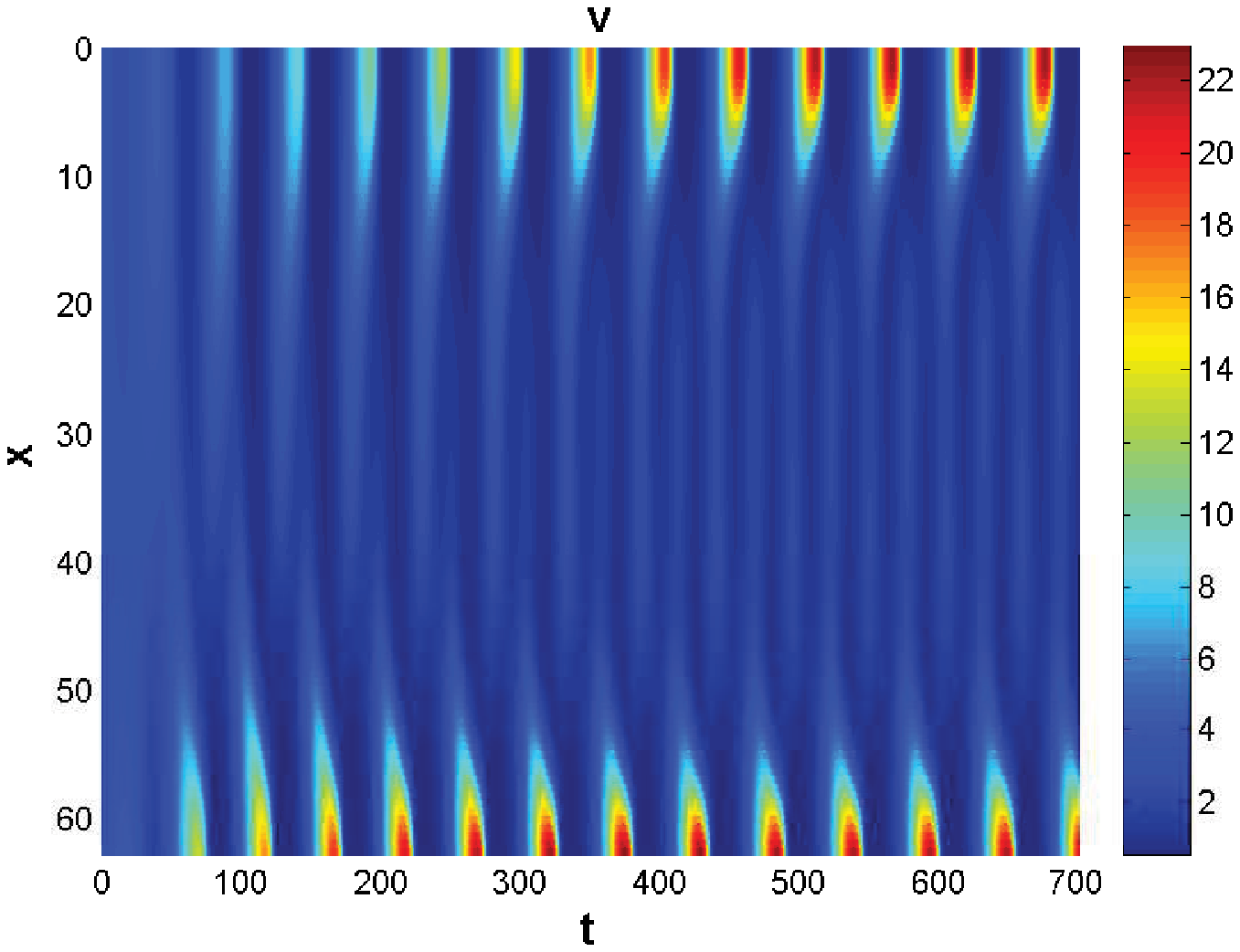}
 \caption{The solution converges to the bifurcating spatially nonhomogeneous periodic solution from $b_{1,+}^H$. Here $b=0.5$, and initial values:  $u(x,0)=3+\ds\f{0.5x^2}{\ell^2\pi^2},~v(x,0)=3+0.5\cos^2x$. (Upper): $\ell=10$;
(Lower): $\ell=20$.\label{hopf2}}
\end{figure}

\end{example}


\end{document}